\documentclass[12pt,reqno]{amsart}
\usepackage{graphicx}
\usepackage{amsmath,amssymb,amsthm}
\usepackage[shortlabels]{enumitem}
\usepackage[hmargin=1in,vmargin=1in]{geometry}
\usepackage[unicode,breaklinks=true,colorlinks=true]{hyperref}
\usepackage[dvipsnames]{xcolor}
\usepackage{color}
\usepackage{cite}

\usepackage[normalem]{ulem} 
\usepackage{tikz}
\usepackage{comment}
\usepackage{psfrag}
\usetikzlibrary{patterns}
\usepackage{float}
\usepackage{epsfig} 

\newtheorem{thm}{Theorem}[section]
\newtheorem{lem}[thm]{Lemma}

\newtheorem{defn}[thm]{Definition}

\newtheorem{assum}[thm]{Assumption}

\theoremstyle{definition}
\newtheorem{eg}[thm]{Example}
\newtheorem{obs}[thm]{Remark}

\newcommand{\tnum}{\rm(\roman*)}
\newcommand{\rnum}{\rm(\alph*)}

\numberwithin{equation}{section}

\def\k0{\kappa_0}

\def\bu{\mathbf{u}}

\def\bu{{\bf{u}}}
\def\bF{{\bf{f}}}
\def\bk{{\bf{k}}}
\def\bx{{\bf{x}}}
\def\b0{{\bf{0}}}
\def\Bs{B_{\rm s}}

\newcommand{\hilite}[1]{{\color{blue} #1}}
\renewcommand{\hilite}[1]{#1}

\begin{document}
\title[Intrinsic expansion approach to the Galerkin approximations for the {NSE}]
{An intrinsic expansion approach to the Galerkin approximations for the Navier--Stokes equations (with an Appendix by Chengzhang Fu)}

\author[L. Hoang]{Luan Hoang$^{1,*}$}
\address{$^1$Department of Mathematics and Statistics,
Texas Tech University\\
1108 Memorial Circle, Lubbock, TX 79409--1042, U. S. A.}
\email{luan.hoang@ttu.edu}

\author[M. S. Jolly]{Michael S. Jolly$^2$}
\address{$^2$Department of Mathematics\\
Indiana University\\ Bloomington, IN 47405, U. S. A.}
\email{msjolly@iu.edu}

\author[(for an appendix) C. Fu]{ (for an Appendix) Chengzhang Fu$^3$}
\address{$^3$Department of Mathematics \& Statistics\\
McMaster University\\ Hamilton, ON L8S 1C7, Canada}
\email{fuc49@mcmaster.ca}

\date{\today}
\thanks{M.S.J. was supported by Simons grant MP-TSM-00002337.  The computational component was supported in part by Lilly Endowment, Inc., through its support for the Indiana University Pervasive Technology Institute.}

\dedicatory{Dedicated to Professor Roger Temam on the occasion of his 85th birthday.}
\thanks{$^*$Corresponding author.}

\begin{abstract}
We study the Galerkin approximation of the three-dimensional  Navier--Stokes equations.  In particular, we examine the convergence of these solutions in a sequence of finite dimensional spaces as the dimension goes to infinity. For any sequence of steady state or, respectively, time dependent Galerkin solutions that converges to a solution of the  Navier--Stokes equations, we obtain a subsequence with an intrinsic asymptotic expansion in appropriate nested function spaces. Consequently, an induced asymptotic expansion is obtained in a more standard  spatial Sobolev or, respectively, spatiotemporal Sobolev--Lebesgue space. In the case of steady states, we establish certain relations among leading terms of this expansion. 
\end{abstract}

\maketitle

\tableofcontents


\section{Introduction} \label{intro}

It is well known that from solutions to Galerkin approximations, one can extract a subsequence that converges (in certain topologies) to a 
 weak solution of the Navier--Stokes equations (NSE) \cite{CF88,TemamAMSbook,FMRTbook}.
There are other ways in which solutions to Galerkin approximations (hereafter referred to as Galerkin solutions) have been useful in the study of the NSE. In \cite{HeywoodRannacher} Heywood and Rannacher infer global existence of smooth solutions of the NSE from stability of nearby Galerkin solutions.
Constantin, Foias and Temam provide in \cite{CFTGalerkin} sufficient conditions on long time behavior of the Galerkin solution for the existence of a nearby steady state of the NSE (see also \cite{TitiGalerkin} for a practical form of these conditions).  Otherwise, little is known about the {\it nature} of the convergence of Galerkin solutions to 
NSE solutions. 

In this paper we study the convergence of Galerkin solutions through a  recently developed theory of asymptotic expansions \cite{FHJ,HJ1}.
The first paper \cite{FHJ} introduced an asymptotic expansion in  normed spaces  in the form
\begin{equation}\label{ex0}
v_n=v+\Gamma_{1,n}w_1+\Gamma_{2,n}w_2+\cdots+\Gamma_{k,n}w_k+\cdots,
\end{equation}
where positive numbers $\Gamma_{k,n}$ and their quotients $\Gamma_{k+1,n}/\Gamma_{k,n}$ converge to zero as $n\to\infty$.
This is referred to as an \emph{intrinsic expansion} in \cite{HJ1} for two reasons. If \eqref{ex0} is viewed as an expansion in $w_k$, then $w_k$ are base functions, and  they depend on the sequence $v_n$. 
If \eqref{ex0} is viewed as an expansion in  $\Gamma_{k,n}$, then  $\Gamma_{k,n}$ are base decaying rates, and  they also depend on $v_n$. 

This notion was used in \cite{FHJ} to study solutions for a sequence of Grashof numbers tending to infinity for any {\it fixed} Galerkin approximation of the NSE.  It made explicit use of fact that the phase space was finite dimensional.  
The development of \eqref{ex0} was then extended in \cite{HJ1} to include certain degenerate types (when some $w_k=0$) and, more critically to expansions in nested infinite dimensional spaces. It was then applied to the study of steady state solutions of the NSE itself also for Grashof numbers going to infinity.   

In this paper we fix the Grashof number and instead study intrinsic expansion of sequences of Galerkin solutions as the dimension of the projected space goes to infinity. We establish the existence of this expansion for both periodic and no-slip boundary conditions in three dimension and for time dependent solutions as well as steady states.   In the steady state case 
the expansion is used to make rigorous connections between the leading terms for the solution and nonlinear term.

We should distinguish the intrinsic expansion of Galerkin solutions from Gevrey regularity.  
This is easiest to see in the periodic case on $[0,2\pi]^2$ with the solution to the NSE and its Galerkin approximation written as a Fourier series 
$$
v(\bx,t)=\sum_{\bk\in \mathbb{Z}^d} \hat{v}_\bk(t) e^{i \bk\cdot \bx}, \qquad 
v_n(\bx,t)=\sum_{\bk\in \mathbb{Z}^d} \hat{v}_{\bk,n}(t) e^{i \bk\cdot \bx}\;.$$
We rewrite 
$$
v_n(\bx,t)=v(\bx,t)+\sum_{\bk\in \mathbb{Z}^d} (\hat{v}_{\bk,n}(t)-\hat{v}_\bk(t)) e^{i \bk\cdot \bx} \;.$$
After ordering the wave vectors on the integer lattice as $\bk_1,\bk_2, \bk_3, \ldots$ as one does the eigenvalues (with repetition),
it might seem natural to identify $\Gamma_{j,n}$ as $|\hat{v}_{\bk_j,n}(t)-\hat{v}_{\bk_j}(t)|$ and $w_j$ as $e^{i\bk_j\cdot\bx}$. After all it is well known that 
$|\hat v_\bk(t)| \lesssim e^{-\lambda|\bk|}$, where $\lambda$ is the radius of spatial analyticity, provided the body force enjoys the same property.  Yet this does not imply that 
$$
\frac{|\hat{v}_{\bk_{j+1},n}(t)-\hat{v}_{\bk_{j+1}}(t)|}{|\hat{v}_{\bk_j,n}(t)-\hat{v}_{\bk_j}(t)|} \to 0, \quad \text{as} \quad n \to \infty$$
even when $|\bk_{j+1}|>|\bk_{j}| $. So $e^{i\bk_j\cdot\bx}$ does not serve as a universal base function $w_j$ in \eqref{ex0} for all convergent sequences of Galerkin solutions, underscoring the dependence of the $w_j$ vectors on each sequence. 

In \hilite{Section} \ref{informal}, we present the main results in simplified forms. They serve as samples of our more technical results in \hilite{Sections} \ref{steadysec} and \ref{timesec} later.
For steady state solutions, we establish an intrinsic expansion in Theorem \ref{E1a} and its properties in Theorem \ref{E3b}.
For time dependent solutions, we obtain in Theorem \ref{tsoln0} an intrinsic expansion.
In subsection \ref{NSEsubsec}, we review the functional setting for the NSE, the Stokes and bilinear operators, (spectral) fractional Sobolev spaces and their compact embeddings. Also presented are some basic inequalities that will be needed later.
We recall in \hilite{Section} \ref{nestedsec} the definitions of strict expansions in nested spaces as well as intrinsic expansions in a normed space.
We state the main Theorems \ref{mainlem} and \ref{maincor} regarding their existence from our previous work.
Example \ref{Pnex} shows that, in general, $v_n=P_n v$ has a degenerate expansion in nested Sobolev spaces. This makes sense of an asymptotic expansion for the forcing term in the Galerkin approximation \eqref{steadyu}.
In \hilite{Section} \ref{steadysec}, we apply the general results in \hilite{Section} \ref{nestedsec} to  study the Galerkin approximations to the steady states of the NSE. 
In Theorem \ref{E1}, we obtain a strict expansion in nested Sobolev spaces and also a unitary or degenerate expansion in a Sobolev space. We explore properties of the obtained expansions in Theorems \ref{EEE} and \ref{E4}. These two theorems reflect two different methods of finding asymptotic expansions for the nonlinear terms. We treat many cases and scenarios based on the asymptotic relations among sequences of coefficients in the expansions obtained for all terms in the NSE.
We establish in \hilite{Section} \ref{timesec} the intrinsic expansion for time dependent solutions, see Theorem \ref{tsoln}. The main idea is the compact embedding in Lemma \ref{compact} for fractional spatial and temporal Sobolev spaces.
It is a natural generalization of a previous result \cite[Chapter 3, Theorem 2.2]{TemamAMSbook} which in turn is a modification of the ubiquitous Aubin--Lions compactness. 
The advantage of Lemma \ref{compact} is that it characterizes the compact embedding by the norms of the same type that involve fractional time derivatives. This allows a simple construction of nested spaces with compact embeddings.
Appendix \ref{ProofAB} contains proofs of basic inequalities for the Stokes  and  bilinear operators.
Appendix \ref{compuapx} illustrates the theoretical findings with some computations.
We treat only the three-dimensional case in this paper knowing that 
the two-dimensional case is quite similar, and in many ways simpler.

\hilite{As a note of caution, each asymptotic expansion obtained in this paper for the NSE applies only to some subsequence of the original sequence of Galerkin solutions.  Its coefficients and vectors as well as their relations heavily depend on the choice of this particular subsequence.}

\section{Main results} \label{informal}
The incompressible NSE for velocity $\bu(\bx,t)$, pressure $p(\bx,t)$ with a  body force $\bF(\bx,t)$, for the spatial variable $\bx$ \hilite{in a domain} in $\mathbb R^d$ ($d=3$) and time variable $t\ge 0$, are 

\begin{equation}\label{NSE}
\left\{ 
\begin{aligned}
&\frac{\partial \bu}{\partial t} -
\nu \Delta \bu  + (\bu\cdot\nabla)\bu + \nabla p = \bF , \\
&\text {div} \ \bu = 0,
\end{aligned}
\right.
\end{equation}
where $\nu>0$ is the kinematic viscosity.
\hilite{ We treat both the no-slip boundary condition in a bounded domain as well as periodic boundary conditions.}
 
It is well known that the initial boundary value problem for \eqref{NSE} can be written in the functional form 
\begin{equation} \label{fNSE}
\frac{{\rm d} u}{{\rm d} t}+ \nu Au+B(u,u)=f \text{ in }V', \quad u(0)=u_0  \in H,
\end{equation}
with $u(t)=\mathbf u(\cdot,t)$,  $f(t)=\mathbf f (\cdot,t)$  and $\bF\in L^2_{\rm loc }([0,\infty),H)$. 
See subsection \ref{NSEsubsec} below for details of the function spaces $H$, $V$, $D(A^\alpha)$, the Stokes operator $A$, and the bilinear operator $B(\cdot,\cdot)$.
For now, we note that $A$ has positive eigenvalues 
$$\lambda_1\le \lambda_2\le \lambda_3\le \ldots \text{ with } \lim_{j\to\infty}\lambda_j=\infty.$$  
There is a (complete) orthonormal basis  $\{\varphi_j:j\in\mathbb N\}$ of $H$ with each $\varphi_j$ being an 
 eigenfunction  of $A$ corresponding to $\lambda_j$.
By the standard scaling of $x$ and $t$, we can assume that $\nu=1$ and 
$\lambda_1 = 1$.

For $n\in\mathbb N$, let $P_n$ denote the projector from $L^2(\Omega)^d$ to the span of $\{\varphi_1, \ldots, \varphi_n\}$.
The Galerkin approximation of \eqref{fNSE} is 
\begin{equation}\label{Galerkinu}
\frac{{\rm d}u_n}{{\rm d}t}+ Au_n+ P_nB(u_n,u_n)=P_nf, \quad   u_n(0)=P_n u_0 ,\quad u_n(t)\in P_n H.
\end{equation}

\hilite{Our main results are in Theorems \ref{E1a}, \ref{E3b} and \ref{tsoln0}. Recall that the asymptotic expansions and their properties are obtained for certain subsequences of the original sequence of the Galerkin solutions.}

\subsection{Steady state case}
Let $f\in H$. For $n\in\mathbb N$, let $u_n$ be a steady state solution of \eqref{Galerkinu}, i.e., 
\begin{equation}\label{steadyu}
 Au_n+ P_nB(u_n,u_n)=P_nf,\quad u_n\in P_n H.
\end{equation}

\begin{thm}\label{E1a}
Let $s=1$ if $f\in V$, or $s\in [0,1)$ if $f\in H$.
There exists a subsequence $(v_n)_{n=1}^\infty$ of $(u_n)_{n=1}^\infty$ with a unitary or degenerate  expansion, in the sense of Definition \ref{defsimnorm},  of the form 
\begin{equation}\label{b4a}
    v_n\approx v+\sum_{k\in\mathcal N}\Gamma_{k,n}w_k
    \text{ in } D(A^s), 
\end{equation}
\hilite{Above, $\mathcal N=\emptyset$ or $\mathcal N=\{1,2,\ldots N\}$ for some integer $N\ge 1$ or $\mathcal N=\mathbb N$, $w_k\in D(A^s)$ and $\Gamma_{k,n}>0$ for all $k\in \mathcal N$ and $n\ge 1$.} 
\end{thm}

Theorem \ref{E1a} is a special case of Theorem \ref{E1}(ii)  below (with $\alpha_*=0$ or $1/2$). \hilite{Briefly speaking, the expansion \eqref{b4a} means that
\begin{align*}
    \|w_k\|_{D(A^s)}&=1,\quad 
    \left\|v_n-\left(v+\sum_{j=1}^{k-1}\Gamma_{j,n} w_j\right)\right\|_{D(A^s)} =\mathcal O(\Gamma_{k,n}) \text{ as }n\to\infty,\\
    \lim_{n\to\infty}\Gamma_{k,n}&=0\quad\text{ and } \quad
    \lim_{n\to\infty} \frac{\Gamma_{k,n}}{\Gamma_{k-1,n}}=0.
\end{align*}

}
 
Now assume $f\in V\setminus\{0\}$. We use the asymptotic  expansion \eqref{b4a} with $s=1$.
A similar expansion is constructed in \hilite{Section} \ref{steadysec} for a subsequence $v_n=u_{\varphi(n)}$ involving the nonlinear term
\begin{equation}\label{PnBexa}
    P_{\varphi(n)}B(v_n,v_n) \approx B(v,v)+\sum_{k\in\mathcal N_1} \rho_{j,n} b_j \; \text{ in } H.
\end{equation} 
\hilite{The expansion \eqref{PnBexa} has the same meaning as \eqref{b4a} with $\mathcal N_1$ being a subset of $\mathbb N$ that plays the role of $\mathcal N$ in \eqref{b4a}.}

 \begin{defn}\label{deforder}
For sequences of positive numbers
 $(\xi_n)_{n=1}^{\infty}$ and 
$(\eta_n)_{n=1}^{\infty}$, we
write  $\xi_n \succ\eta_n$  if $\xi_n/\eta_n \to \infty$ as $n \to
\infty$, and  $\xi_n \sim \eta_n$ if  $\xi_n/\eta_n \to \lambda \in (0,\infty)$.  We write $\xi_n \succsim \eta_n$ if either  $\xi_n\succ\eta_n$ or $\xi_n \sim\eta_n$.
A set  $X$ of sequences of positive numbers is called \emph{totally comparable} if it holds that for any sequences $(\xi_n)_{n=1}^\infty,(\eta_n)_{n=1}^\infty\in X$, either $\xi_n\succsim \eta_n$ or  $\eta_n\succ \xi_n$.  
 \end{defn}

This allows us to compare the coefficients in the asymptotic expansions \eqref{b4a} and \eqref{PnBexa} and  deduce rigorous relations between leading vectors.

\begin{thm}\label{E3b} Assume that $v_n \ne v$ and $Q_nf\ne 0$ for all $n$.
Let $G_{n}=\|Q_n f\|_{L^2}$. Assume the set $\{(\Gamma_{j,n})_{n=1}^\infty, (\rho_{k,n})_{n=1}^\infty,(G_n)_{n=1}^\infty:j\in\mathcal N,k\in \mathcal N_1\}$ is totally comparable. One of the following cases must hold.

\noindent Case 1: ($b_1$ in \eqref{PnBexa} exists) and 
    $(\Gamma_{1,n}\succ G_n \text { or } \rho_{1,n}\succ G_n)$.
\begin{enumerate}[label=\tnum]
    \item Case $\Gamma_{1,n}\sim \rho_{1,n}$.  Then one has 
        \begin{equation*}
    Aw_1+\mu_0 b_1=0\text{ with }  \mu_0=\lim_{n\to\infty} \frac{\rho_{1,n}}{\Gamma_{1,n}}>0.
    \end{equation*}
In addition, let $w_2$, $b_2$ exist and Assumption \ref{chiassum} hold for  $\beta=0$. 
Then either there are $\mu_1,\mu_2,\mu_3\in\mathbb R$, with at least  one of them being $1$ and ($\mu_1\mu_3>0$ or $\mu_1=\mu_3=0$), such that
\begin{equation*}
    \mu_1 Aw_2+\mu_2 b_1 +\mu_3 b_2=0,
\end{equation*}
   or we have 
    \begin{equation*}
        \|v_n-v-\Gamma_{1,n}w_1\|_{D(A)}=\mathcal O(\lambda_{\varphi(n)+1}^{-1/2})\text{ as }n\to\infty.
    \end{equation*}

\item Case $\Gamma_{1,n}\succ \rho_{1,n}$. Then $w_1=0$ and one has the degenerate expansion
\begin{equation*}
    v_n\approx v +\sum_{k\in\mathbb N}\Gamma_{k,n}\cdot 0 \text{ in }D(A).
\end{equation*}

\item Case $\rho_{1,n}\succ \Gamma_{1,n}$. Then $b_1=0$ and  one has the degenerate expansion
\begin{equation*}
    P_{\varphi(n)}B(v_n,v_n)\approx B(v,v) +\sum_{k\in\mathbb N}\rho_{k,n}\cdot 0 \text{ in }H.
\end{equation*}
\end{enumerate}

\noindent Case 2:  $b_1$ in \eqref{PnBex} does not exist, or otherwise, 
($G_n\succsim \Gamma_{1,n}$ and $G_n\succsim \rho_{1,n}$). 
Then 
\begin{equation*}
    \|v_n-v\|_{D(A)}=\mathcal  O(\lambda_{\varphi(n)+1}^{-1/2})\text{ as }n\to\infty.
\end{equation*}
\end{thm}

Theorem \ref{E3b} is a special case of Theorem \ref{EEE} below with $\alpha_*=1/2$, $s=1$ and $\beta=0$.
    
\subsection{Time dependent solutions}\label{maintime}

 Let $T>0$, $u_0\in H$ and $f\in L^2(0,T;H)$. Let $u_n$, for $n\in\mathbb N$, be a solution of \eqref{Galerkinu}.
Let $(v_n)_{n=1}^\infty$ be a subsequence of $(u_n)_{n=1}^\infty$  such that $v_n$ converges in $L^2(0,T;V)$ weakly and in $L^2(0,T;H)$ strongly to a  weak solution $u$ of the NSE \eqref{fNSE}. The existence of such a subsequence is standard, see e.g. \cite{CF88,TemamAMSbook,FMRTbook}.

\begin{thm}\label{tsoln0}
    There exists a subsequence of $v_n$ (also denoted $v_n$) that possesses a unitary or degenerate expansion, in the sense of Definition \ref{defsimnorm},
    \begin{equation}\label{vtex0}
        v_n\approx u+\sum_{k\in \mathcal N_2} \bar \Gamma_{k,n} \bar w_k\text{ in } L^2(0,T;H),
    \end{equation}
  \hilite{ for some set $\mathcal N_2$ of the same type as $\mathcal N$ in \eqref{b4a}.}
\end{thm}

Theorem \ref{tsoln0} is a simple consequence of  Theorem \ref{tsoln} below and will be proved near the end of \hilite{Section} \ref{timesec}.

\subsection{Navier--Stokes preliminaries}\label{NSEsubsec}
We present some basic facts about the NSE.
Two types of boundary conditions are considered for the NSE.

\begin{itemize}
    \item \textit{Periodic boundary condition.} We consider the NSE with periodic boundary conditions in $\Omega =[0,L]^d$.
Let $\mathcal V$ denote the set of $\mathbb R^d$-valued $\Omega$-periodic divergence-free trigonometric polynomials with zero average over $\Omega$.
Define the spaces 
$$\text{$H=$ closure of $\mathcal V$ in $L^2(\Omega)^d$ and $V=$ closure of $\mathcal V$ in $H^1(\Omega)^d$.}$$ 
\item \textit{No-slip boundary condition.} 
Consider the NSE in an open, connected, bounded set $\Omega$ in $\mathbb R^d$. For the sake of simplicity, we assume the boundary $\partial \Omega$ is of class   $C^\infty$. Let $\mathbf n$ denote the outward normal vector on the boundary. 
Under the no-slip boundary condition 
$$\bu =0\text{ on }\partial \Omega,$$
the relevant spaces are
\begin{equation*} 
H=\{\bu\in L^2(\Omega)^d:\, \nabla \cdot \bu=0, \, \bu\cdot\mathbf n|_{\partial \Omega}=0\}\text{ and } 
V=\{\bu\in H_0^1(\Omega)^d:\, \nabla \cdot \bu=0 \}.
\end{equation*}
\end{itemize}

The inner product and its associated norm in $H$ are those of $L^2(\Omega)^d$ and are denoted by $\langle\cdot,\cdot\rangle$ and $|\cdot|$, respectively.
(We will also use $|\cdot|$ for the modulus of a vector in 
$\mathbb{C}^d$; we assume that the 
meaning will be clear from the context.)

The space $V$ is equipped with the following inner product and norm
\begin{equation}\label{Vnorm}
\langle\!\langle u,v \rangle\!\rangle =\int_{\Omega}\sum_{j=1}^{d} 
\frac{\partial}{\partial x_j}u(\bx)\cdot 
\frac{\partial}{\partial x_j}v(\bx)  {\rm d} \bx
\text{ and }
\|u\|=\langle\!\langle u,u \rangle\!\rangle^{1/2} \text{ for }u,v\in V.
\end{equation}
Note that the norm $\|\cdot\|$ in $V$ above is equivalent to the $H^1(\Omega)^d$-norm.

We use the following embeddings  and identification for $H$, $V$ and their dual spaces $H'$, $V'$
\begin{equation}\label{emiden}
    V\subset H=H'\subset V'.
\end{equation}
For convenience, we define the product $\langle\cdot,\cdot\rangle_{V',V}$ on $V'\times V$ by   $$\langle w,v\rangle_{V',V}=w(v) \text{ for }w\in V', v\in V.$$
For $u\in H$ and $v\in V$, using \eqref{emiden}, we have $\langle u,v\rangle_{V',V}=\langle u,v\rangle$.

Define the bounded linear mapping $A:V\to V'$ and continuous bilinear mapping $B:V\times V\to V'$ by
$$\langle Au,v\rangle_{V',V}=\langle\!\langle u,v \rangle\!\rangle, \quad 
\langle B(u,v),w\rangle_{V',V}=\int_\Omega ((u(\bx)\cdot \nabla )v(\bx))\cdot w(\bx) {\rm d} \bx,$$
for $u,v,w\in V$.
We also denote
$$\Bs(u,v)=B(u,v)+B(v,u).$$

Let $\mathcal{P}$ be the Helmholtz-Leray projection onto $H$, 
that is, the orthogonal projection of $L^2(\Omega)^d$ onto $H$. 
Define the space $D(A)=V\cap H^2(\Omega)^d$ equipped with the inner product $\langle Au, Av\rangle$. Then $D(A)$ is a (real) Hilbert space and its induced norm $\|u\|_{D(A)}=|Au|$ is, in fact, equivalent to the $H^2(\Omega)^d$-norm.
We have 
\begin{align}\label{AP}
    Au&=-\mathcal P\Delta u \in H \text{ for } u\in D(A),\\
\label{BP}
B(u,v)&=\mathcal{P}\left( (u \cdot \nabla) v \right)\in H,\text{ for } u\in D(A),v\in V.
\end{align}
The Stokes operator is the restriction $A:D(A)\to H$.

Recall the orthogonality relation of the bilinear term 
(see for instance \cite{T97})
\begin{equation}\label{Bflip}
\langle B(u,v),w\rangle = -\langle B(u,w),v\rangle,  \text{ so that }
\langle B(u,v),v\rangle=0\text{ for } u,v,w\in D(A),
\end{equation}

We now recall the fractional operators and spaces.
For $\alpha\ge 0$, define the fractional Sobolev space
$$ D(A^\alpha)=\left \{ u=\sum_{n=1}^\infty c_n \varphi_n\in H: \sum_{n=1}^\infty \lambda_n^{2\alpha} |c_n|^2<\infty\right\}$$ 
and the fractional operator
$$ A^\alpha u =\sum_{n=1}^\infty \lambda_n^\alpha c_n \varphi_n \text{ for }
u=\sum_{n=1}^\infty c_n \varphi_n\in D(A^\alpha).$$
In fact, $D(A^\alpha)$ is a Hilbert space with the inner product and induced norm
$$\langle u,v\rangle_{D(A^\alpha)}=\langle A^\alpha u,A^\alpha v\rangle
\text{ and } \|u\|_{D(A^\alpha)}=|A^\alpha u|\text{ for } u,v\in D(A^\alpha).$$
Clearly, $D(A^0)=H$, $D(A^1)=D(A)$, and $D(A^{1/2})=V$ with 
\begin{equation}\label{VH1} 
\|u\|_{D(A^{1/2})}=\|u\|=|A^{1/2}u|.
\end{equation}
Moreover,  we have, thanks to the fact $\lambda_1=1$,
$$\|u\|_{D(A^\alpha)}\ge \|u\|_{D(A^\beta)} \text{ for }\alpha\ge \beta\ge 0, \ u\in D(A^\alpha).$$
In particular,
$$|Au|\ge \|u\|  \text{ for  }u\in D(A),\quad  \|u\|\ge |u| \text{ for  }u\in V.$$

It is well known that, for $\alpha>\beta\ge 0$,  
\begin{equation}\label{comp0}
    \text{the space $D(A^\alpha)$ is compactly embedded into $D(A^\beta)$.}
\end{equation}

In the  periodic case, by using \cite[(2.5), Lemma 2.1]{HM1} (with the Gevrey order $\sigma=0$), we have, for any number $\alpha\ge 0$ and $u,v\in D(A^{\alpha+1/2})\cap D(A)$,
\begin{equation*}
    |A^\alpha B(u,v)|\le K_\alpha (\|u\|^{1/2}|Au|^{1/2}|A^{\alpha+1/2} v|+|A^{\alpha+1/2} u| |A^{3/4}v|),
\end{equation*}
for some positive constant $K_\alpha$.
Since
$|A^{3/4}v|=\langle A^{1/2}v,Av \rangle^{1/2} \le \|v\|^{1/2} |Av|^{1/2},$
it follows that 
\begin{equation}\label{ABper}
    |A^\alpha B(u,v)|\le K_\alpha (\|u\|^{1/2}|Au|^{1/2}|A^{\alpha+1/2} v|+|A^{\alpha+1/2} u|\,  \|v\|^{1/2} |Av|^{1/2}).
\end{equation}
In particular, when $u=v$,
\begin{equation}\label{Aau}
    |A^\alpha B(u,u)|\le 2K_\alpha \|u\|^{1/2}|Au|^{1/2}|A^{\alpha+1/2} u|
    \text{ for } u\in D(A^{\alpha+1/2})\cap D(A).
\end{equation}

In the  no-slip case, we have $\varphi_j\in C^\infty(\bar \Omega)^d$. We recall inequality (1.48) in \cite[Remark 1.6, Chapter 1, page 13]{TemamAMSbook}. For $m\in\mathbb N\cup \{0\}$, there is a positive constant $C$ such that
\begin{equation} \label{Preg}
\|\mathcal Pu\|_{H^m(\Omega)^d}\le C \|u\|_{H^m(\Omega)^d} \text{ for all }u\in H^m(\Omega)^d.
\end{equation}
Then for any $m\in \mathbb N\cup\{0\}$, there exist positive constants $C$ and $C'$ such that
\begin{equation}\label{AHm}
     C'\|u\|_{H^m(\Omega)^d} \le |A^{m/2}u|\le C\|u\|_{H^m(\Omega)^d} \text{ for all }u\in D(A^{\hilite{m}/2}).
\end{equation}
Regarding the bilinear operator $B(\cdot,\cdot)$, an alternative of \eqref{ABper} is the following.
For any $m\in\mathbb N$, there is a positive constant $\bar K_m$ such that 
\begin{equation}\label{AjB}
    |A^{m/2} B(u,v)|\le  \bar K_m |A^{(m+1)/2}u|\, |A^{(m+1)/2}v|
    \text{ for any }u,v\in D(A^{(m+1)/2}).
\end{equation}

The first inequality  in \eqref{AHm} was proved in, e.g., \cite[(4.25) and (4.27)]{CF88}. The inequality \eqref{AjB}, after utilizing \eqref{AHm}, is the same as \cite[(1.8)]{FS87} which was stated without proof.
For the sake of completeness, the proofs the second inequality in \eqref{AHm} and inequality \eqref{AjB} are provided in Appendix \ref{ProofAB}.

\section{Background on intrinsic asymptotic expansions}\label{nestedsec}
In this section, we recall the main definitions and results on intrinsic asymptotic expansions from \cite{HJ1}.

\begin{defn}\label{uexp}
Let $(Z_k,\|\cdot\|_{Z_k})$ for $k\ge 0$ be normed spaces over $\mathbb K=\mathbb C$ or $\mathbb R$ such that
\begin{equation}\label{nested}
Z_0\subset Z_1 \subset Z_2 \subset \ldots \subset Z_k \subset Z_{k+1} \subset\ldots 
\end{equation}
with continuous embeddings. Denote $\mathcal Z=(Z_k)_{k=0}^\infty$.

    We say a sequence $(v_n)_{n=1}^\infty$ in $Z_0$  has a \emph{strict expansion} in $\mathcal Z$ if it satisfies one of the following three conditions.
\begin{enumerate}[label={\rm (\Roman*)}]
    \item\label{d1} There exists $v\in Z_0$ such that  $v_n=v$ for all $n\ge 1$.

    \item\label{d2} There are integer $K\ge 1$, vector $v\in Z_0$,  vectors $w_k\in Z_k$,  vectors $w_{k,n}\in Z_{k-1}$ and positive numbers $\Gamma_{k,n}$, for $n\ge 1$ and $1\le k\le K$, such that 
\begin{align}
\label{b0}
\lim_{n\to\infty} \Gamma_{1,n}&=0,\\
\label{b1}
\lim_{n\to\infty} \frac{\Gamma_{k+1,n}}{\Gamma_{k,n}}&=0 
\text{ for all   $1\le k< K$,}
\end{align}
\begin{align}
\label{b3}
\lim_{n\to\infty} \|w_{k,n}-w_k\|_{Z_k}&=0  \text{ for all $1\le k\le K$,}\\
\label{b2}
\|w_{k,n}\|_{Z_{k-1}}&=1  \text{ for all $n\ge 1$, $1\le k\le K$,}
\end{align}
\begin{equation}\label{b4}
    v_n=\underbrace{v+\sum_{j=1}^{k-1} \Gamma_{j,n}w_j}_{\text{partial sum}}+\underbrace{\Gamma_{k,n}w_{k,n}}_
    {\text{remainder}}\text{ for all $n\ge 1$, $1\le k\le K$, }
\end{equation}
and
\begin{equation*}
    w_{K,n}=w_{K} \text{ for all $n\ge 1$.}
\end{equation*}

   \item\label{d3}  There are vector $v\in Z_0$,  vectors $w_k\in Z_k$, vectors $w_{k,n}\in Z_{k-1}$, positive integers $N_k$  and numbers $\Gamma_{k,n}> 0$ for all $n\ge 1$, $k\ge 1$  such that one has \eqref{b0}, while the limits in  \eqref{b1} and \eqref{b3}  hold for all $k\ge 1$, the equation in \eqref{b4} holds for all $n,k\ge 1$, and  the unit vector condition in  \eqref{b2} holds for all $k\ge 1$, $n\ge N_k$.
\end{enumerate}

We say the expansion in Case \ref{d1} is trivial, the expansions in Cases \ref{d1} and \ref{d2} are finite, and 
the expansion in Case \ref{d3} is infinite.
\end{defn}

We denote the strict expansion in Definition \ref{uexp} by  
\begin{equation}\label{vsum}
    v_n\approx v+\sum_{k\in\mathcal N} \Gamma_{k,n} w_k,
\end{equation}
where the set $\mathcal N$ is empty in Case \ref{d1}, is $\{1,2,\ldots,K\}$ in Case \ref{d2}, and is $\mathbb N$ in Case \ref{d3}.

\begin{obs}\label{remscl}
We recall some facts from \cite[Remark 3.2]{HJ1}.
Assume the strict expansion \eqref{vsum}.
\begin{enumerate}[label=\rnum]
\item We have
\begin{equation}\label{vlim}
    \lim_{n\to\infty} \|v_n-v\|_{Z_0}=0
\end{equation}
and
\begin{equation*}
  \frac{  \left \|  v_n-\left(v+\sum_{j=1}^{k-1} \Gamma_{j,n}w_j\right)\right\|_{Z_k}}{\Gamma_{k-1,n}}
  =\frac{\Gamma_{k,n}}{\Gamma_{k-1,n}}\| w_{k,n}\|_{Z_k}\to 0\cdot \|w_k\|_{Z_k}=0\text{ as }n\to\infty,
\end{equation*}
which justify the term ``expansion".

\item Consider the case $\mathcal N\ne \emptyset$. Then  
\begin{equation}\label{Gamze}
\lim_{n\to\infty} \Gamma_{k,n}=0 \text{ for all }k\in\mathcal N.
\end{equation}
We denote
\begin{equation}\label{Nbar}
 \bar N_k=\begin{cases}
     1, & \text{ in Case \ref{d2} for } 1\le k\le K,\\
     \max\{N_1,N_2,\ldots,N_k\}, & \text{ in Case \ref{d3} for all } k\ge 1.
 \end{cases}
 \end{equation}
Then the unit vector property \eqref{b2} holds true for $k\ge 1$ and $n\ge \bar N_k$.

\item Any subsequence $(v_{n_j})_{j=1}^\infty$ of $(v_n)_{n=1}^\infty$ 
has the strict expansion
\begin{equation}\label{vsubsum}
   v_{n_j}\approx v+\sum_{k\in\mathcal N} \Gamma_{k,n_j} w_k
\end{equation}
which is of the same case \ref{d1}, \ref{d2} or \ref{d3} as  $(v_n)_{n=1}^\infty$.
\end{enumerate}
\end{obs}

By virtue of {\cite[Proposition 3.4]{HJ1}}
the three cases \ref{d1}, \ref{d2}, \ref{d3} in Definition \ref{uexp} are exclusive.
Moreover, according to \cite[Proposition 3.5]{HJ1},
the  strict expansion \eqref{vsum} is asymptotically unique, that is, the set $\mathcal N$ on the right-hand side of  \eqref{vsum} is unique, the vectors $v$ and $w_k$ are uniquely determined for $k\in \mathcal N$, while 
 the vectors   $w_{k,n}$ and positive numbers $\Gamma_{k,n}$ are uniquely determined for $k\in \mathcal N$ and $n\ge \bar N_k$, where $\bar N_k$ is defined by \eqref{Nbar}.

The notion of a strict expansion was first developed in \cite {FHJ} for a fixed Galerkin approximation of the NSE, where the sequence $(v_n)_{n=1}^\infty$ was for Grashof numbers $G_n\to \infty$.  In that work, each space $Z_k=\mathbb R^N$ for some fixed $N$.  The extension to the infinite dimensional case of the full NSE (also as $G\to \infty$) was done in \cite{HJ1}, assuming the embeddings are compact. This, of course, is valid regardless how the sequence $(v_n)_{n=1}^\infty$ is generated. 

\begin{thm}[{\cite[Theorem 3.7]{HJ1}}]\label{mainlem} 
Let  $\mathcal Z=((Z_k,\|\cdot\|_{Z_k}))_{k=0}^\infty$ be a family of normed spaces over $\mathbb K=\mathbb C$ or $\mathbb R$ such that \eqref{nested} holds with all embeddings being compact. Then any convergent sequence in $Z_0$ has a subsequence that possesses a strict expansion in $\mathcal Z$.
\end{thm}

\begin{defn}\label{refinex}
Let  $\mathcal Z=(Z_k)_{k=0}^\infty$ be as in Definition \ref{uexp} and $(v_n)_{n=1}^\infty$ be a sequence in $Z_0$.
\begin{enumerate}[label=\tnum]
    \item   We say $(v_n)_{n=1}^\infty$ has a \emph{relaxed expansion} if it satisfies the same conditions in Definition \ref{uexp} except that the unit vector condition \eqref{b2}  is removed, and, 
    \begin{equation*}
        w_K\ne 0\text{ in Case \ref{d2}.}
    \end{equation*}

    \item\label{uni}   We say $(v_n)_{n=1}^\infty$ has a \emph{unitary expansion} if it has a relaxed expansion and in Case \ref{d2}, respectively, Case \ref{d3}, it requires that
    \begin{equation}\label{wkone}
        \|w_k\|_{Z_k}=1\text{ for all $1\le k\le K$, respectively,  $k\ge 1$.}
    \end{equation}
    
\item\label{degenx}  We say $(v_n)_{n=1}^\infty$ has a \emph{degenerate expansion} if it satisfies the same conditions as an infinite unitary expansion in part \ref{uni} except that the requirement \eqref{wkone} is replaced with the following: either 
    \begin{enumerate}[label=\rnum]
        \item $w_k=0$ for all $k\ge1$, or 
        \item there exists an integer $N\ge 1$ such that
            \begin{equation}\label{wktwo}
                \|w_k\|_{Z_k}=1\text{ for all $1\le k\le N$, and $w_k=0$ for all $k>N$.}
            \end{equation}
    \end{enumerate} 
\end{enumerate}  
\end{defn}
We will use the same notation in \eqref{vsum} to denote all expansions introduced above  but will clearly specify their types when used.

\begin{defn}\label{defsimnorm}
    Let $(Z,\|\cdot\|_Z)$ be a normed space over $\mathbb C$ or $\mathbb R$ and $(v_n)_{n=1}^\infty$ be a sequence in $Z$. The strict expansion, resp., unitary expansion and degenerate expansion, with 
$$\mathcal Z=((Z_k,\|\cdot\|_{Z_k}))_{k=0}^\infty := ((Z,\|\cdot\|_Z))_{k=0}^\infty$$ is called the strict unitary expansion, resp., unitary expansion and degenerate expansion   in $Z$.
\end{defn}

In Definition \ref{defsimnorm},  each $Z_k=Z$, the norm on the remainder $w_{k,n}$ in \eqref{b2} is the same as the norm on the limit
$w_k$ in \eqref{wktwo}, hence the terminology {\it strict unitary expansion} is used as in \cite{FHJ}. Clearly, if \eqref{vsum} is a strict unitary or unitary expansion in $Z$, then $\|w_k\|_Z=1$ for all $k$.

Same as Remark \ref{remscl}(c), if $v_n$ has an asymptotic expansion of any type in Definitions \ref{refinex} and \ref{defsimnorm}, then any subsequence of $v_n$ has the same expansion (of the same type).

The main result on the existence of the above intrinsic expansions is the following.
\begin{thm}[{\cite[Corollary 3.12]{HJ1}}]\label{maincor}
Let $\mathcal Z=((Z_k,\|\cdot\|_{Z_k}))_{k=0}^\infty$  be as in Theorem \ref{mainlem}, and $(Z,\|\cdot\|)$ be a normed space over $\mathbb K$ such that
 all $Z_k\subset Z_\infty$ with continuous embeddings. 
Then any convergent sequence in $Z_0$ has a subsequence that possesses either a unitary expansion or a degenerate expansion in both 
\begin{enumerate}[label=\rnum]
    \item\label{caseq}  a subsequence $\widetilde{\mathcal Z}$  of $\mathcal Z$, and
    \item\label{casingle} space $Z_\infty$.
\end{enumerate}
\end{thm}

Elementary examples for $Z=\mathbb R$, $\mathbb C$ in \cite[subsection 4.3]{FHJ},
for $Z=H$ in \cite[Example 3.13]{HJ1}, and for $Z=V$ with a sequence of steady states of the NSE in \cite[Example 4.5]{HJ1} have expansions with some nonzero vectors $w_k$. In contrast, the next example provides an asymptotic expansion, 
for a sequence $v_n=P_n g$ where $g$ is fixed, with each $w_k=0$.

\begin{eg}\label{Pnex}
Let $(\alpha_k)_{k=0}^\infty$ be a strictly decreasing sequence of  positive numbers.
Suppose $g\in D(A^{\alpha_0})$ with $Q_n g\ne 0$ for all $n$. 

(a) For $k\ge 1$ and $n\ge 1$, let 
\begin{equation*}
\widetilde  \Gamma_{k,n}=\|Q_n g\|_{D(A^{\alpha_{k-1}})},\quad w_k=0,\quad \widetilde w_{k,n}=-Q_n g/\widetilde \Gamma_{k,n}.
\end{equation*} 
We have, for any $k\ge 1$ and $n\ge 1$, that
\begin{equation*}
P_n g =g-Q_ng=g+\widetilde \Gamma_{k,n}\widetilde w_{k,n}.
\end{equation*}
Note that
    $\widetilde \Gamma_{1,n}=\|Q_n g\|_{D(A^{\alpha_0})} \to 0$ as $n\to\infty$, and 
\begin{equation*} 
\frac{\widetilde \Gamma_{k+1,n}}{\widetilde \Gamma_{k,n}}=\frac{\|Q_n g\|_{D(A^{\alpha_k})}}{\|Q_n g\|_{D(A^{\alpha_{k-1}})}}
    \le \lambda_{n+1}^{\alpha_k-\alpha_{k-1}} \to 0\text{ as }n\to\infty.
\end{equation*}
Also,     $\|\widetilde w_{k,n}\|_{D(A^{\alpha_{k-1}})}=1$, and 
\begin{equation*}
   \|\widetilde w_{k,n}\|_{D(A^{\alpha_k})}=\frac{\|Q_n g\|_{D(A^{\alpha_k})}}{\|Q_n g\|_{D(A^{\alpha_{k-1}})}}
    \to 0\text{ as }n\to\infty,
\end{equation*}
that is $\widetilde w_{k,n}\to 0=w_k$ in $D(A^{\alpha_k})$.
Thus, we obtain the degenerate expansion 
\begin{equation}\label{degenPn}
     P_n g\approx g+\sum_{k=1}^\infty \widetilde \Gamma_{k,n}w_k= g+\sum_{k=1}^\infty \widetilde \Gamma_{k,n}\cdot 0\text{ in }\mathcal Z=(D(A^{\alpha_k}))_{k=0}^\infty.
 \end{equation}

(b) Let $\beta=\inf\{\alpha_k : k\in\mathbb N\}$ and $s$ be any number in the interval $[0,\beta]$.
Let  
$$\Gamma_{k,n}=\|Q_n g\|_{D(A^{\alpha_k})}, \quad  w_{k,n}=-Q_n g/\Gamma_{k,n}.$$
We have $P_n g =g+\Gamma_{k,n} w_{k,n}$. 
Same as part (a), we have $\Gamma_{1,n} \to 0$ and  
$\Gamma_{k+1,n}/\Gamma_{k,n}\to 0$ as $n\to\infty$, and 
\begin{equation*}
   \|w_{k,n}\|_{D(A^s)}=\frac{\|Q_n g\|_{D(A^s)}}{\|Q_n g\|_{D(A^{\alpha_k})}}\le \lambda_{n+1}^{s-\alpha_k}
    \to 0\text{ as }n\to\infty.
\end{equation*}
Then one has the degenerate expansion
\begin{equation}\label{degen3}
     P_n g\approx g+\sum_{k=1}^\infty \Gamma_{k,n}w_k= g+\sum_{k=1}^\infty \Gamma_{k,n}\cdot 0 \text{ in } D(A^s).
 \end{equation}
\end{eg}

The above degenerate expansions \eqref{degenPn} and \eqref{degen3}  hint at having all terms in \eqref{Galerkinu} expanded asymptotically in $n$. This is the main idea behind our approach.

As in \cite{FHJ} the following notion is needed to make our study rigorous when using the comparison in Definition \ref{deforder}.
Let $X$ be a set of sequences of positive numbers
and $(\varphi(n))_{n=1}^\infty$ be a subsequence of $(n)_{n=1}^\infty$.
Define
$$X_\varphi=\left \{ (\xi_{\varphi(n)})_{n=1}^\infty \text{ with } (\xi_n)_{n=1}^\infty \in X\right \}.$$ 
We call $X_\varphi$ a \emph{subsequential set} of $X$.
Note that if $X$ is totally comparable, then so is the susequential set $X_\varphi$.

\section{Stationary solutions} \label{steadysec}
Let $f\in H$. For $n\ge 1$, let $u_n$ be a solution of the steady state equation \eqref{steadyu}. 
First, we derive some estimates for these $u_n$.

\begin{lem}\label{R1}
Consider a real number $\alpha\ge 0$ in the  periodic case, and $\alpha=m/2$ for some integer $m\ge 0$ in the  no-slip case. Assume   $f\in D(A^\alpha)$. Then  there is a number  $M_\alpha>0$ depending on $|A^\alpha f|$ such that
\begin{equation}\label{regest}
\|u_n\|_{D(A^{\alpha+1})}\le M_\alpha\quad\forall \ n\in\mathbb N.
\end{equation}    
\end{lem}
\begin{proof}
Using \eqref{Bflip} and Cauchy's inequality leads to 
\begin{equation}\label{Vbound}
    \|u_n\|\le |f|.
\end{equation}

Taking the inner product of equation \eqref{steadyu} with $Au_n$ in $H$ gives
\begin{equation}\label{Aueq}
    |Au_n|^2=-\langle P_nB(u_n,u_n),Au_n\rangle+\langle P_nf, Au_n\rangle
    =-\langle (u_n\cdot\nabla) u_n,Au_n\rangle+\langle f, Au_n\rangle.
\end{equation}
From  H\"older's inequality with $L^\infty,L^3,L^2$, Agmon's inequality $\|u_n\|_{L^\infty} \le c_1 \|u_n\|_{H^1}^{1/2}\|u_n\|_{H^2}^{1/2}$ and \eqref{AHm}, 
one obtains
\begin{align*}
    |\langle (u_n\cdot\nabla) u_n,Au_n\rangle|
    &\le \|u_n\|_{L^\infty} \|\nabla u_n\|_{L^2} |Au_n| 
    \le c_2 (\|u_n\|^{1/2} |Au_n|^{1/2}) \|u_n\|\cdot |Au_n|\\
    &= c_2 \|u_n\|^{3/2} |Au_n|^{3/2}.
\end{align*}
Combining this with \eqref{Aueq} and Young's inequality gives
\begin{align*}
    |Au_n|^2\le c_2 \|u_n\|^{3/2} |Au_n|^{3/2} +|f||Au_n|
    &\le \frac{1}{2}|Au_n|^2+c_3(\|u_n\|^6 +|f|^2).
\end{align*}
Thus, together with \eqref{Vbound}, it follows that
\begin{equation}\label{H2est}
    |Au_n|^2\le 2c_3(\|u_n\|^6 +|f|^2)\le M_0:=2c_3(|f|^6+|f|^2).
\end{equation}
This proves \eqref{regest} for $\alpha=0$. 

We prove for the case $\alpha>0$ below by iterations. Let  $\beta\ge 1/2$.
Taking the inner product of equation \eqref{steadyu} with $A^{2\beta}u_n$ in $H$ gives
\begin{align*}
    |A^{\beta+1/2}u_n|^2
    &=\langle Au_n,A^{2\beta}u_n\rangle
    = -\langle B(u_n,u_n),A^{2\beta}u_n\rangle+\langle f,A^{2\beta} u_n\rangle \\
    &= -\langle A^{\beta-1/2} B(u_n,u_n),A^{\beta+1/2}u_n\rangle+\langle A^{\beta-1/2} f,A^{\beta+1/2} u_n\rangle.
\end{align*}
By applying the Cauchy--Schwarz and Cauchy inequalities, we obtain 
\begin{equation}\label{inest}
    |A^{\beta+1/2}u_n|^2
    \le |A^{\beta-1/2} B(u_n,u_n)|\, |A^{\beta+1/2}u_n|+ \frac14 |A^{\beta+1/2}u_n|^2+|A^{\beta-1/2} f|^2.
\end{equation}

\medskip\noindent
\textit{The periodic case.}  
Applying Ineq.~\eqref{Aau} to $\alpha:=\beta-1/2$ and then using Cauchy's inequality yield
\begin{align*}
|A^{\beta-1/2} B(u_n,u_n)|\, |A^{\beta+1/2}u_n|
&\le (2K_{\beta-1/2} \|u_n\|^{1/2}|Au_n|^{1/2}|A^\beta u_n|)\cdot |A^{\beta+1/2}u_n|\\
&\le \frac14 |A^{\beta+1/2}u_n|^2 + 4K_{\beta-1/2}^2\|u_n\|\cdot |Au_n|\cdot |A^\beta u_n|^2.
\end{align*}
Combining this with \eqref{inest} and then  \eqref{Vbound}, \eqref{H2est} yields 
\begin{equation}\label{betaboot}
    |A^{\beta+1/2}u_n|^2\le 8K_{\beta-1/2}^2\|u_n\|\cdot |Au_n|\cdot |A^\beta u_n|^2 +2|A^{\beta-1/2} f|^2
    \le \bar C_\beta |A^\beta u_n|^2 +2|A^{\beta-1/2} f|^2,
\end{equation}
where $\bar C_\beta=8K_{\beta-1/2}^2 \hilite{M_0^{1/2}}|f|$.

Now let $\alpha\in(0,\infty)$. Then there are $\bar \beta\in(1/2,1]$ and $k\in\mathbb N$ such that $\alpha+1=\bar\beta+k/2$. We apply estimate \eqref{betaboot} consecutively  to $\beta:=\bar\beta+j/2$ for $j=0,1,\ldots,k-1$, and also use \eqref{H2est} to estimate $|A^{\bar \beta}u_n|$ for the initial application. 
These result in estimate \eqref{regest}.

\medskip\noindent
\textit{The no-slip case.} 
For $j\ge 2$, applying inequality \eqref{AjB} and then Cauchy's inequality yields
\begin{align*}
|A^{j/2-1/2} B(u_n,u_n)|\, |A^{j/2+1/2}u_n|
    &\le \bar K_{j-1} |A^{j/2}u_n|^2 \cdot |A^{j/2+1/2}u_n|\\
   & \le \frac14 |A^{j/2+1/2}u_n|^2 +  \bar K_{j-1}^2 |A^{j/2}u_n|^4.
\end{align*}
Combining this inequality with \eqref{inest} for $\beta=j/2$ gives
\begin{equation}\label{Ajboot}
|A^{(j+1)/2}u_n|^2\le 2\bar K_{j-1}^2 |A^{j/2}u_n|^4 +2|A^{(j-1)/2} f|^2.
\end{equation}

Consider $\alpha=m/2$ with $m\in\mathbb N$. By applying estimate \eqref{Ajboot} to $j=2,3,\ldots,m,m+1$ and using the estimate \eqref{H2est} to start with, we obtain \eqref{regest}.
\end{proof}

\subsection{Existence of the intrinsic expansions}\label{Gaexpand}
Let $(\varphi(n))_{n=1}^\infty$ be a subsequence of $(n)_{n=1}^\infty$.
For $n\ge 1$, let $v_n=u_{\varphi(n)}$, that is, $v_n \in P_{\varphi(n)}H$ satisfies
\begin{equation}\label{steady}
Av_n+ P_{\varphi(n)}B(v_n,v_n)=P_{\varphi(n)}f.
\end{equation}

\begin{thm}\label{E1}
Consider a real number $\alpha_*\ge 0$ in the  periodic case, or $\alpha_*=m/2$ for some nonnegative integer $m$ in the  no-slip case.
Assume   $f\in D(A^{\alpha_*})\setminus\{0\}$.
\begin{enumerate}[label=\tnum]
    \item\label{E1i} Given any strictly decreasing sequence $(s_k)_{k=0}^\infty$ of numbers in the interval $(0,\alpha_*+1)$, set $Z_k=D(A^{s_k})$ and $\mathcal Z=(Z_k)_{k=0}^\infty$.
Then there is a subsequence of $(v_n)_{n=1}^\infty$ that has a strict expansion in $\mathcal Z$.    

\item\label{E1ii} Let $s$ be any number in the interval $[0,\alpha_*+1)$.
Then there is a subsequence of $(v_n)_{n=1}^\infty$ that has a unitary expansion or degenerate expansion in $D(A^s)$.    
\end{enumerate}
\end{thm}
\begin{proof}
(i) Let $Z_*=D(A^{\alpha_*+1})$. By the estimate \eqref{regest} for $\alpha=\alpha_*$, the sequence $v_n$ is bounded in $Z_*$. By the compact embedding of $Z_*$ into $Z_0$, which is due to \eqref{comp0}, there is a subsequence of $(v_n)_{n=1}^\infty$, still denoted by $(v_n)_{n=1}^\infty$, such that $v_n\to v$  in $Z_0$.
Using \eqref{comp0} again to have the compact embedding $Z_k\Subset Z_{k+1}$, we apply Theorem \ref{mainlem}.

(ii) Choose $(s_k)_{k=0}^\infty$ as in part (i) with $s_k>s$ for all $k\ge 0$.
After having a subsequence of $(v_n)_{n=1}^\infty$, still denoted by $(v_n)_{n=1}^\infty$, with $v_n\to v$ in $Z_0$, we apply Theorem \ref{maincor}\ref{casingle} to $Z_\infty=D(A^s)$.
\end{proof}

\subsection{Properties of the  expansions}\label{more}
Consider a real number $\alpha_*> 0$ in the  periodic case, or $\alpha_*=m/2$ for some $m\in\mathbb N$ in the  no-slip case. Assume   $f\in D(A^{\alpha_*})\setminus\{0\}$.  
For the remainder of this section, let $\beta$ be a fixed number in $[0,\alpha_*)$, which in the no-slip case is additionally assumed  to be half of an integer.

By applying Theorem \ref{E1}\ref{E1ii} to $s=1+\beta$, we obtain a unitary or degenerate expansion for a subsequence of $v_n$, but here still denoted by $v_n$,
\begin{equation}\label{Vx}
v_n\approx v+\sum_{k\in\mathcal N} \Gamma_{k,n}w_k \text{ in } D(A^{1+\beta}).  
\end{equation}
Taking the limit in $H$ of equation \eqref{steady} and using the continuity of $A:D(A)\to H$ and $B:D(A)\times D(A)\to H$ we obtain 
\begin{equation}\label{e0}
Av+B(v,v)=f,
\end{equation}
that is, $v$ is a steady state solution.
Next, we present in subsections \ref{met1} and \ref{met2} two ways one can proceed to find more information about the asymptotic expansion \eqref{Vx}.

Since we are not interested in the trivial case $v_n=v$ for all $n$,  we will focus on the situation $ v_n \ne v$  for all $n$ for the rest of this section.

\subsubsection{First method: Direct expansion for the nonlinear terms}\label{met1}
By the virtue of Lemma \ref{R1}, the sequence $v_n$ is bounded in $D(A^{\alpha_*+1})$, hence $P_{\varphi(n)}B(v_n,v_n)$ is bounded in $D(A^{\alpha_*})$. Same as Theorem \ref{E1}\ref{E1ii}, we have, after using a subsequence, a unitary or degenerate expansion
\begin{equation}\label{PnBex}
    P_{\varphi(n)}B(v_n,v_n)\approx  b_0+\sum_{k\in \mathcal N_1} \rho_{k,n} b_k \text{ in }D(A^\beta).
\end{equation} 
Note that the above expansion \eqref{PnBex} of $P_{\varphi(n)}B(u_n,u_n)$ does not use the asymptotic expansion \eqref{Vx} of  $v_n$.

Thanks to \eqref{vlim}, we have
\begin{equation}\label{Pb0}
     P_{\varphi(n)}B(v_n,v_n)\to b_0 \text{ in $D(A^\beta)$ as $n\to\infty$.}
\end{equation}
Because $v_n\to v$ in $D(A^{1+\beta})$ we have $Av_n\to Av$ in $D(A^\beta)$ and $P_{\varphi(n)}B(v_n,v_n)\to B(v,v)$ in $D(A^\beta)$.
Combining the last property with \eqref{Pb0} gives 
\begin{equation}\label{b0uu}
    b_0=B(v,v).
\end{equation}

\hilite{To derive the relations between the vectors of the expansions \eqref{Vx} and \eqref{PnBex}, we ``substitute" their finite approximations into the equation \eqref{steady}. 
After some simplifications, the resulting equations such as \eqref{releq3} and \eqref{releq4} will contain certain coefficient sequences. There is a need to compare all of them. This is done by selecting an appropriate subsesequence which results in technical requirements in \eqref{Stotal} and Assumptions \ref{chiassum} and \ref{newSa} below.}

Consider the case $Q_nf\ne 0$ for all $n$ first.
For $n\ge 1$, define $G_n=\|Q_{\varphi(n)}f\|_{D(A^\beta)}$. Note that
\begin{equation}\label{Gnrate}
    G_n\le \lambda_{\varphi(n)+1}^{\beta-\alpha_*} \|f\|_{D(A^{\alpha_*})}=\mathcal O(\lambda_{\varphi(n)+1}^{\beta-\alpha_*})\text{ as }n\to\infty.
\end{equation}
Let 
\begin{equation}\label{Sdef}
\mathcal S=\{(\Gamma_{j,n})_{n=1}^\infty, (\rho_{k,n})_{n=1}^\infty,(G_n)_{n=1}^\infty:j\in\mathcal N,k\in \mathcal N_1\}.
\end{equation}
Thanks to the fact $\mathcal S$ is countable and 
 \cite[Lemma 3.2]{FHJ}, $\mathcal S$ has a subsequential set that is totally comparable.
Together with \hilite{property (c) of Remark \ref{remscl}}, we assume 
\begin{equation}\label{Stotal}
    \text{ the set $\mathcal S$ itself  is totally comparable.}
\end{equation}

\begin{assum}\label{chiassum}
In the case $Q_nf\ne 0$ for all $n$, and $b_1$ in \eqref{PnBex} exists with $\Gamma_{1,n}\sim \rho_{1,n}$, we define 
\begin{equation}\label{xn}
  \mu_0=\lim_{n\to\infty} \frac{\rho_{1,n}}{\Gamma_{1,n}}\in(0,\infty)\text{ and set }\chi_n=\rho_{1,n}/\Gamma_{1,n}-\mu_0.
\end{equation}
Assume that the sequence $(\chi_n)_{n=1}^\infty$ satisfies
either
\begin{enumerate}
\item [{\rm (S1)}] $\chi_n=0$ for all $n$, or

\item [{\rm (S2)}] ($\chi_n>0$ for all $n$) or ($\chi_n<0$ for all $n$). In addition, whenever $w_2$ and $b_2$ exist, the set 
\begin{equation*}
    \{  (\Gamma_{2,n})_{n=1}^\infty, (\rho_{2,n})_{n=1}^\infty,  (\Gamma_{1,n}|\chi_n|)_{n=1}^\infty, (G_n)_{n=1}^\infty\} \text{ is totally comparable.} 
\end{equation*}
\end{enumerate}
\end{assum}

By the virtue of \cite[Lemma A.4]{FHJ} and using a subsequential set, Assumption \ref{chiassum} can always be met.

\begin{thm}\label{EEE} Assume that $v_n \ne v$ and $Q_nf\ne 0$ for all $n$.

\noindent Case 1: ($b_1$ in \eqref{PnBex} exists) and 
\begin{equation}\label{as1}
    (\Gamma_{1,n}\succ G_n \text { or } \rho_{1,n}\succ G_n).
\end{equation}
\begin{enumerate}[label=\tnum]
    \item Case $\Gamma_{1,n}\sim \rho_{1,n}$.  Then one has 
        \begin{equation}\label{sharea}
    Aw_1+\mu_0 b_1=0\text{ with }  \mu_0=\lim_{n\to\infty} \frac{\rho_{1,n}}{\Gamma_{1,n}}>0.
    \end{equation}
 Set $\chi_n=\rho_{1,n}/\Gamma_{1,n}-\mu_0$. Let $w_2$, $b_2$ exist and Assumption \ref{chiassum} hold.
 \begin{enumerate}[label=\rnum]
     \item If 
     \begin{equation}\label{ggx}
        \text{  $\Gamma_{2,n}\succ G_n $ or $\rho_{2,n} \succ G_n$ or, in the case $\chi_n\ne 0$,  $\Gamma_{1,n}|\chi_n|\succ G_n $, }
            \end{equation}
then there are $\mu_1,\mu_2,\mu_3\in\mathbb R$ with at least  one of them being $1$ and ($\mu_1\mu_3>0$ or $\mu_1=\mu_3=0$), such that
\begin{equation}\label{share2a}
    \mu_1 Aw_2+\mu_2 b_1 +\mu_3 b_2=0.
\end{equation}
    \item Otherwise, we have 
    \begin{equation}\label{vvGam}
        \|v_n-v-\Gamma_{1,n}w_1\|_{D(A^{1+\beta})}=\mathcal O(\lambda_{\varphi(n)+1}^{\beta-\alpha_*})\text{ as }n\to\infty.
    \end{equation}
 \end{enumerate}

\item Case $\Gamma_{1,n}\succ \rho_{1,n}$. Then $w_1=0$ and one has the degenerate expansion
\begin{equation}\label{vndeg}
    v_n\approx v +\sum_{k\in\mathbb N}\Gamma_{k,n}\cdot 0 \text{ in }D(A^{1+\beta}).
\end{equation}

\item Case $\rho_{1,n}\succ \Gamma_{1,n}$. Then $b_1=0$ and  one has the degenerate expansion
\begin{equation}\label{Bvndeg}
    P_{\varphi(n)}B(v_n,v_n)\approx B(v,v) +\sum_{k\in\mathbb N}\rho_{k,n}\cdot 0 \text{ in }D(A^\beta).
\end{equation}
\end{enumerate}

\noindent Case 2:  $b_1$ in \eqref{PnBex} does not exist, or otherwise, 
($G_n\succsim \Gamma_{1,n}$ and $G_n\succsim \rho_{1,n}$). 
Then 
\begin{equation}\label{vnrate}
    \|v_n-v\|_{D(A^{1+\beta})}=\mathcal O(\lambda_{\varphi(n)+1}^{\beta-\alpha_*})\text{ as }n\to\infty.
\end{equation}
\end{thm}
\begin{proof}
Consider Case 1. We have from \eqref{Vx}, \eqref{PnBex} and \eqref{b0uu} that
\begin{equation}\label{vbn}
    v_n=v+\Gamma_{1,n}w_{1,n}\text{ and } P_{\varphi(n)}B(v,v)=B(v,v)+\rho_{1,n}b_{1,n}
\end{equation}
with $w_{1,n}\to w_1$ in $D(A^{1+\beta})$ and $b_{1,n}\to b_1$ in $D(A^\beta)$ as $n\to\infty$.
Substituting the identities in \eqref{vbn} into \eqref{steady} gives
 \begin{equation*}
 Av+\Gamma_{1,n}Aw_{1,n}+ B(v,v)+\rho_{1,n}b_{1,n}=f-Q_{\varphi(n)}f,
 \end{equation*}
which, together with \eqref{e0}, implies
\begin{equation}
 \Gamma_{1,n}Aw_{1,n}+\rho_{1,n}b_{1,n}=-Q_{\varphi(n)}f. \label{releq3}   
\end{equation}

(i) Note from \eqref{as1} and the fact $\Gamma_{1,n}\sim \rho_{1,n}$ that 
\begin{equation}\label{qq}
 \frac{\|Q_{\varphi(n)}f\|_{D(A^\beta)}}{\Gamma_{1,n}}= \frac{G_n}{\Gamma_{1,n}}\to 0\text{ and }
  \frac{\|Q_{\varphi(n)}f\|_{D(A^\beta)}}{\rho_{1,n}}= \frac{G_n}{\rho_{1,n}}\to 0\text { as }
  n\to\infty.
\end{equation}
Dividing \eqref{releq3} by $\Gamma_{1,n}$ and taking the limit of the resulting equation in $D(A^\beta)$ as $n\to\infty$ with the use of \eqref{qq}, we obtain \eqref{sharea}.

Let $w_2$, $b_2$ exist and Assumption \ref{chiassum} hold. We have
\begin{equation}\label{wbn}
    \Gamma_{1,n}w_{1,n}=\Gamma_{1,n}w_1+\Gamma_{2,n}w_{2,n} \text{ and } 
    \rho_{1,n}b_{1,n}=\rho_{1,n}b_1+ \rho_{2,n}b_{2,n},
\end{equation} 
with $w_{2,n}\to w_2$ in $D(A^{1+\beta})$ and $b_{2,n}\to b_2$ in $D(A^\beta)$ as $n\to\infty$.
From \eqref{releq3} and \eqref{wbn}, we obtain 
\begin{equation*}
    \Gamma_{1,n}Aw_1+\Gamma_{2,n}Aw_{2,n}+\rho_{1,n}b_1+\rho_{2,n}b_{2,n}=-Q_{\varphi(n)}f.
\end{equation*}
Now, with $\rho_{1,n}=\Gamma_{1,n}(\chi_n+\mu_0)$, see \eqref{xn}, it follows that
\begin{equation*}
\Gamma_{1,n}( Aw_1+\mu_0 b_1)+\Gamma_{2,n}Aw_{2,n}+\Gamma_{1,n}\chi_n b_1+\rho_{2,n}b_{2,n}=-Q_{\varphi(n)}f.
\end{equation*}
 Thanks to \eqref{sharea}, we have $\Gamma_{1,n}( Aw_1+\mu_0 b_1)=0$, 
hence,
  \begin{equation}
   \Gamma_{2,n}Aw_{2,n}+\Gamma_{1,n}\chi_n b_1+\rho_{2,n}b_{2,n}=-Q_{\varphi(n)}f.\label{releq4}
 \end{equation}

(a) Comparing the sequences $(\Gamma_{2,n})_{n=1}^\infty$,  $(\rho_{2,n})_{n=1}^\infty$, and, in the case $\chi_n\ne 0$, $(\Gamma_{1,n}|\chi_n|)_{n=1}^\infty$ with respect to the relation $\succsim$, we can select a sequence $(\xi_n)_{n=1}^\infty$ among them such that 
$\xi_n\succsim  \Gamma_{2,n}$, $\xi_n\succsim \rho_{2,n}$ and, in the case $\chi_n\ne 0$, $\xi_n\succsim \Gamma_{1,n}|\chi_n|$.
Thanks to \eqref{ggx}, we have $G_n/\xi_n\to 0$ as $n\to\infty$.
Dividing \eqref{releq4} by $\xi_n$ if the sequence $(\xi_n)_{n=1}^\infty\in \{(\Gamma_{1,n})_{n=1}^\infty,(\rho_{2,n})_{n=1}^\infty\}$, or by $\xi_n{\rm sign}(\chi_n)\ne 0$ otherwise, and then taking the limit in $D(A^\beta)$  as $n\to\infty$, we obtain \eqref{share2a}.

(b) In this case, from the negation of \eqref{ggx} and \eqref{Sdef}, \eqref{Stotal}, we have $G_n\succsim \Gamma_{2,n}$. Thus,
\begin{equation*}
    \|v_n-v-\Gamma_{1,n}w_1\|_{D(A^{1+\beta})}=\mathcal O(\Gamma_{2,n})=    \mathcal O(G_n)=\mathcal O(\lambda_{\varphi(n)+1}^{\beta-\alpha_*})\text{ as }n\to\infty,
\end{equation*}
where the last estimate comes from \eqref{Gnrate}. 
which proves \eqref{vvGam}.

 (ii) With \eqref{as1}, we have, in this case, $\Gamma_{1,n}\succ G_n$, and hence the first limit in \eqref{qq} holds true.
 Dividing \eqref{releq3} by $\Gamma_{1,n}$ and taking the limit in $D(A^\beta)$ as $n\to\infty$, we obtain $Aw_1=0$, which yields  $w_1=0$. This implies the degenerate expansion \eqref{vndeg}.

 (iii) Same as for (ii),  we have $\rho_{1,n}\succ G_n$ and  the second limit in \eqref{qq}.
 Dividing \eqref{releq3} by $\rho_{1,n}$ and taking the limit in $D(A^\beta)$ give $b_1=0$. Thus, we obtain the degenerate expansion \eqref{Bvndeg}.

\medskip
Consider Case 2. 
Suppose $b_1$ does not exist. Then the expansion \eqref{PnBex} is trivial, that is, $P_{\varphi(n)}B(v_n,v_n)=b_0=B(v,v)$ for all $n$.
Subtracting equations for $v_n$ and $v$ gives
\begin{equation*}
    A(v_n-v)=-Q_{\varphi(n)}f.
\end{equation*}
This equation and \eqref{Gnrate} give
\begin{equation*}
     \|v_n-v\|_{D(A^{1+\beta})}=\|Q_{\varphi(n)}f\|_{D(A^\beta)}=\mathcal O(\lambda_{\varphi(n)+1}^{\beta-\alpha_*})\text{ as }n\to\infty,
\end{equation*}
and hence we obtain \eqref{vnrate}.

Now, suppose $b_1$ exists and \eqref{as1} fails.
Then we have $G_n\succsim \Gamma_{1,n}$. Hence,
\begin{equation*}
    \|v_n-v\|_{D(A^{1+\beta})}=\mathcal O(\Gamma_{1,n})=    \mathcal O(G_n)=\mathcal O(\lambda_{\varphi(n)+1}^{\beta-\alpha_*})\text{ as }n\to\infty.
\end{equation*}
Therefore, we obtain \eqref{vnrate} again.
\end{proof}

Note, in the proof of Case 2 above, that in particular one has the estimate \eqref{vnrate} whenever $b_1$ does not exist.

\begin{thm}\label{Ef1}
Assume $f\in P_{n_*}H\setminus\{0\}$ for some $n_*\in\mathbb N$, and $v_n\ne v$ for all $n$, and $b_1$ exists. 
\begin{enumerate}[label=\tnum]
    \item Case  $\Gamma_{1,n}\sim \rho_{1,n}$.
   Then one has \eqref{sharea}. If, in addition,  Assumption \ref{chiassum} holds with $(G_n)_{n=1}^\infty$ being removed, then there are $\mu_1,\mu_2,\mu_3\in \mathbb R$ with at least  one of them being $1$ and ($\mu_1\mu_3>0$ or $\mu_1=\mu_3=0$), such that \eqref{share2a} holds.
\item Case $\Gamma_{1,n}\succ \rho_{1,n}$. Then one has the degenerate expansion \eqref{vndeg}.

\item Case $\rho_{1,n}\succ \Gamma_{1,n}$. Then one has the degenerate expansion \eqref{Bvndeg}.
\end{enumerate}
\end{thm}
\begin{proof}
    The proof is the same as for parts (i), (ii) and (iii) of Case 1 in Theorem \ref{EEE} with $Q_{\varphi(n)}f=0$ for $n\ge n_*$ and $G_n$ being removed.
\end{proof}

\subsubsection{Second method: Expansion by substitution for the nonlinear terms}\label{met2}
We  make direct connections between $P_{\varphi(n)}B(v_n,v_n)$ and the asymptotic expansion of $v_n$. From here to \eqref{exd2}, the arguments are heuristic only. 
Using the expansion \eqref{Vx} in equation \eqref{steady}, we \textit{formally} write
\begin{multline*}
    Av+\sum_{k\in\mathcal N} \Gamma_{k,n} Aw_k
    +P_{\varphi(n)}B(v,v)+ \sum_{m\in\mathcal N} \Gamma_{m,n} P_{\varphi(n)}\Bs(v,w_m) \\
    +\sum_{m,j\in\mathcal N} \Gamma_{m,n}\Gamma_{j,n}P_{\varphi(n)}B(w_m,w_j)
    =P_{\varphi(n)}f.
\end{multline*}
Using $P_{\varphi(n)}={\rm Id}-Q_{\varphi(n)}$ in the above equation  yields
\begin{multline}\label{exd1}
    Av+\sum_{k\in\mathcal N} \Gamma_{k,n} Aw_k
    +B(v,v)+ \sum_{m\in\mathcal N} \Gamma_{m,n} \Bs(v,w_m) +\sum_{m,j\in\mathcal N} \Gamma_{m,n}\Gamma_{j,n}B(w_m,w_j)
    =f-Q_{\varphi(n)}f
    \\ +Q_{\varphi(n)}B(v,v)+\sum_{m\in\mathcal N} \Gamma_{m,n} Q_{\varphi(n)}\Bs(v,w_m)
     +\sum_{m,j\in\mathcal N} \Gamma_{m,n}\Gamma_{j,n}Q_{\varphi(n)}B(w_m,w_j).
\end{multline}
Thanks \eqref{e0}, $Av$, $B(v,v)$ and $f$ can be removed from \eqref{exd1} and it follows that
\begin{multline}\label{exd2}
 \sum_{k\in\mathcal N} \Gamma_{k,n} Aw_k
    + \sum_{m\in\mathcal N} \Gamma_{m,n} \Bs(v,w_m) +\sum_{m,j\in\mathcal N} \Gamma_{m,n}\Gamma_{j,n}B(w_m,w_j) =-Q_{\varphi(n)}f\\
   +Q_{\varphi(n)}B(v,v)+\sum_{m\in\mathcal N} \Gamma_{m,n} Q_{\varphi(n)}\Bs(v,w_m)
     +\sum_{m,j\in\mathcal N} \Gamma_{m,n}\Gamma_{j,n}Q_{\varphi(n)}B(w_m,w_j).
\end{multline}
The asymptotic expansions in \eqref{exd2} are meant 
to be in $D(A^\beta)$.
Denote
\begin{align}\label{seq1}
    \widetilde\rho_{0,n}&=\|Q_{\varphi(n)}B(v,v) \|_{D(A^\beta)},\
    \widetilde\rho_{m,n}=\|Q_{\varphi(n)}\Bs(v,w_m) \|_{D(A^\beta)},\\
    \widetilde\rho_{m,j,n}&=\|Q_{\varphi(n)}B(w_m,w_j)\|_{D(A^\beta)}, \notag
\end{align}
for $m,j\in\mathcal N$.
Then the sequences of our interest from \eqref{exd2} are 
\begin{equation}\label{allseq}
(\Gamma_{k,n})_{n=1}^\infty,
(\Gamma_{m,n}\Gamma_{j,n})_{n=1}^\infty, \left(G_n\right)_{n=1}^\infty,
\left(\widetilde\rho_{0,n}\right)_{n=1}^\infty, 
\left(\Gamma_{m,n} \widetilde\rho_{m,n}\right)_{n=1}^\infty, \left(\Gamma_{m,n}\Gamma_{j,n}\widetilde\rho_{m,j,n}\right)_{n=1}^\infty.
\end{equation}
for $k,m,j\in\mathcal N$. 
The general idea, again, is to compare the sequences in \eqref{allseq} and use them to collect equivalent terms in \eqref{exd2}.
However, we do not always use all sequences in \eqref{allseq}.
To give a rigorous proof of Theorem \ref{E4} below, we make the following assumption.

\begin{assum}\label{newSa}
Assume $v_n\ne v$ and $Q_{\varphi(n)}f\ne 0$ for all $n$, and the following (S3), (S4).
\begin{enumerate}
    \item[\rm (S3)] For $j=0,1$, either ($\widetilde\rho_{j,n}=0$ for all $n$) or ($\widetilde\rho_{j,n}>0$ for all $n$).
    
    \item[\rm (S4)] First step: let $\mathcal S_0=\mathcal S_1=\mathcal S=\{(\Gamma_{k,n})_{n=1}^\infty,
 \left(G_n\right)_{n=1}^\infty,(\Gamma_{m,n}\Gamma_{j,n})_{n=1}^\infty,:k,m,j\in \mathcal N\}$.
 Second step:  if $\widetilde\rho_{0,n}>0$, then update the set $\mathcal S_1$ to be 
 $\mathcal S_0\cup \{\left(\widetilde\rho_{0,n}\right)_{n=1}^\infty\}$.
 Third step: if $\widetilde\rho_{1,n}>0$,  then update the set $\mathcal S$ to be 
 $\mathcal S_1\cup\{\left(\Gamma_{1,n} \widetilde\rho_{1,n}\right)_{n=1}^\infty\}$.
The resulting set $\mathcal S$ is assumed to be totally comparable.
\end{enumerate}
\end{assum}

We can always choose an appropriate subsequence so that Assumption \ref{newSa} holds. Indeed,  starting with (S4), in the first step, by applying  \cite[Lemma 3.2]{FHJ} and using a subsequential set, we can assume that $\mathcal S_0$ is totally comparable.
In the second step, applying  \cite[Lemma A.4]{FHJ} to the set $\mathcal S_0$ and the sequence $\left(\widetilde\rho_{0,n}\right)_{n=1}^\infty$, we can,  after using a subsequential set,  assume (S3) holds for $j=0$ and that the resulting set $\mathcal S_1$ is totally comparable.
In the third step, applying  \cite[Lemma A.4]{FHJ} to the updated set $\mathcal S_1$ and the sequence $\left(\widetilde\rho_{1,n}\right)_{n=1}^\infty$, we can,  after using a subsequential set,  assume  (S3)  holds for $j=1$ and  that  the resulting set $\mathcal S$ is totally comparable. Therefore, Assumption \ref{newSa} holds for a subsequence of $\varphi(n)$, still denoted by $\varphi(n)$.

Same as \eqref{Gnrate}, we have  from \eqref{seq1} that
\begin{equation}\label{rhrate}
    \widetilde\rho_{0,n}=\mathcal O(\lambda_{\varphi(n)+1}^{\beta-\alpha_*})
    \text{ and }
    \widetilde\rho_{1,n}=\mathcal O(\lambda_{\varphi(n)+1}^{\beta-\alpha_*})\text{ as }n\to\infty.
\end{equation}

\begin{thm}\label{E4}
Under Assumption \ref{newSa}, one has the following cases.

\begin{enumerate}[label=\tnum]
    \item\label{T46i}  If 
    \begin{equation} \label{Gtr1}
    \Gamma_{1,n}\succ G_n \text{ and } \lim_{n\to\infty}  \frac{\widetilde\rho_{0,n}}{\Gamma_{1,n}}=0,
    \end{equation}
    then
\begin{equation}\label{rel1}
    Aw_1+\Bs(v,w_1)=0.
\end{equation}

Assume, in addition, that  $w_2$ exists.

\begin{enumerate}[label=\rnum]
    \item Suppose 
\begin{equation}\label{gam2}
   \left(\Gamma_{2,n}\succ  G_n
   \text{ and } 
   \lim_{n\to\infty}  \frac{\widetilde\rho_{0,n}}{\Gamma_{2,n}}
   =\lim_{n\to\infty}  \frac{\Gamma_{1,n}\widetilde\rho_{1,n}}{\Gamma_{2,n}}=0\right)
\end{equation}
or 
\begin{equation}\label{gam1squ}
  \left(\Gamma_{1,n}^2\succ  G_n
   \text{ and } 
   \lim_{n\to\infty}  \frac{\widetilde\rho_{0,n}}{\Gamma_{1,n}^2}
   =\lim_{n\to\infty}  \frac{\widetilde\rho_{1,n}}{\Gamma_{1,n}}=0\right).
\end{equation}

\begin{itemize}
    \item If $ \Gamma_{2,n}\succsim  \Gamma_{1,n}^2$, then 
\begin{equation}\label{rel2}
    Aw_2+\Bs(v,w_2)+\lambda B(w_1,w_1)=0, \text{ where }\lambda=\lim_{n\to\infty}\frac{\Gamma_{1,n}^2}{\Gamma_{2,n}}\ge 0.
\end{equation}
    \item If $\Gamma_{1,n}^2\succ  \Gamma_{2,n}$, then 
\begin{equation}\label{rel3}
     B(w_1,w_1)=0.
\end{equation}
\end{itemize}

\item Otherwise,
\begin{equation}\label{rate3}
     \|v_n-v\|_{D(A^{1+\beta})}=\mathcal O(\lambda_{\varphi(n)+1}^{(\beta-\alpha_*)/2})\text{ as }n\to\infty,
\end{equation}
\begin{equation}\label{rate4}
     \|v_n-v-\Gamma_{1,n}w_1\|_{D(A^{1+\beta})}=\mathcal O(\lambda_{\varphi(n)+1}^{\beta-\alpha_*})\text{ as }n\to\infty.
\end{equation}
\end{enumerate}

\item If \eqref{Gtr1} fails, then
\begin{equation}\label{rate1}
    \|v_n-v\|_{D(A^{1+\beta})}=\mathcal O(\lambda_{\varphi(n)+1}^{\beta-\alpha_*})\text{ as }n\to\infty.
\end{equation}
\end{enumerate}
\end{thm}
\begin{proof}
We have
$$Av+\Gamma_{1,n}Aw_{1,n}+P_{\varphi(n)} B(v+\Gamma_{1,n}w_{1,n},v+\Gamma_{1,n}w_{1,n}) =f-Q_{\varphi(n)}f.$$
By the bilinearity of $B$, we can expand as
$$Av+\Gamma_{1,n}Aw_{1,n}+P_{\varphi(n)} B(v,v)
+\Gamma_{1,n} P_{\varphi(n)} \Bs(v,w_{1,n})+\Gamma_{1,n}^2P_{\varphi(n)}B(w_{1,n},w_{1,n})=P_{\varphi(n)}f.$$
Using $P_{\varphi(n)} ={\rm Id}-Q_{\varphi(n)} $ to rewrite  $P_{\varphi(n)} B(v,v)$ and $P_{\varphi(n)}f$, we have
\begin{align*}
&Av+\Gamma_{1,n}Aw_{1,n}+B(v,v)-Q_{\varphi(n)}B(v,v)
+\Gamma_{1,n} P_{\varphi(n)} \Bs(v,w_{1,n})\\
&\quad +\Gamma_{1,n}^2P_{\varphi(n)}B(w_{1,n},w_{1,n})
=f-Q_{\varphi(n)}f.    
\end{align*}
Applying \eqref{e0}, we remove $Av$, $B(v,v)$, $f$ and reach
\begin{equation}\label{firstex}
\Gamma_{1,n}Aw_{1,n}
+\Gamma_{1,n} P_{\varphi(n)} \Bs(v,w_{1,n})+\Gamma_{1,n}^2P_{\varphi(n)}B(w_{1,n},w_{1,n})=-Q_{\varphi(n)}f+Q_{\varphi(n)}B(v,v).    
\end{equation}
Note in \eqref{firstex} that, thanks to \eqref{Gamze}, $\Gamma_{1,n}\succ \Gamma_{1,n}^2$.

(i)  Under the assumption \eqref{Gtr1}, by dividing equation \eqref{firstex} by $\Gamma_{1,n}$ and taking the limit in $D(A^\beta)$ as $n\to\infty$, we obtain \eqref{rel1}.

(a) We assume \eqref{gam2} or \eqref{gam1squ} now. Using the substitution 
$\Gamma_{1,n}w_{1,n}=\Gamma_{1,n}w_1+\Gamma_{2,n}w_{2,n}$ for the first two terms in \eqref{firstex}, we have
\begin{align*}
&\Gamma_{1,n}Aw_1 + \Gamma_{2,n}Aw_{2,n}
+\Gamma_{1,n} P_{\varphi(n)} \Bs(v,w_1)
+\Gamma_{2,n} P_{\varphi(n)} \Bs(v,w_{2,n})\\
&+\Gamma_{1,n}^2P_{\varphi(n)}B(w_{1,n},w_{1,n})=-Q_{\varphi(n)}f+Q_{\varphi(n)}B(v,v).
\end{align*}
Now, writing  $ P_{\varphi(n)} \Bs(v,w_1)=\Bs(v,w_1)-Q_{\varphi(n)} \Bs(v,w_1)$, we obtain
\begin{equation*}
\Gamma_{1,n}(Aw_1+\Bs(v,w_1))+\Gamma_{2,n}(Aw_{2,n}+P_{\varphi(n)}\Bs(v,w_{2,n})
+\Gamma_{1,n}^2P_{\varphi(n)}B(w_{1,n},w_{1,n})=\mathcal Q_n, 
\end{equation*}
where 
\begin{equation*}
    \mathcal Q_n=-Q_{\varphi(n)}f+Q_{\varphi(n)}B(v,v)+\Gamma_{1,n}Q_{\varphi(n)} \Bs(v,w_1).
\end{equation*}
    Thanks to \eqref{rel1}, we reduce this to 
\begin{equation}\label{eq3}
    \Gamma_{2,n}(Aw_{2,n}+P_{\varphi(n)}\Bs(v,w_{2,n})\hilite{)}
+\Gamma_{1,n}^2P_{\varphi(n)}B(w_{1,n},w_{1,n})=\mathcal Q_n.
\end{equation}    

Let $\xi_n=\Gamma_{2,n}$ in the case $ \Gamma_{2,n}\succsim  \Gamma_{1,n}^2$, 
and $\xi_n=\Gamma_{1,n}^2$ in the case $\Gamma_{1,n}^2\succ  \Gamma_{2,n}$.
In either case \eqref{gam2} or \eqref{gam1squ}, we have 
\begin{equation}\label{xGr}
   \xi_n\succ  G_n
   \text{ and } 
   \lim_{n\to\infty}  \frac{\widetilde\rho_{0,n}}{\xi_n}
   =\lim_{n\to\infty}  \frac{\Gamma_{1,n}\widetilde\rho_{1,n}}{\xi_n}=0.
\end{equation}
Using the facts in \eqref{xGr} to estimate the $D(A^\beta)$-norm of $\mathcal Q_n$ on the right-hand side of equation \eqref{eq3}, we obtain the following estimate, as $n\to\infty$,
    $$\|\Gamma_{2,n}(Aw_{2,n}+P_{\varphi(n)}\Bs(v,w_{2,n})+\Gamma_{1,n}^2P_{\varphi(n)}B(w_{1,n},w_{1,n})\|_{D(A^\beta)}=o(\xi_n).$$
Dividing the last equation by $\xi_n$ and passing $n\to\infty$, we obtain \eqref{rel2} and \eqref{rel3}.

(b) Combining the  negation of [\eqref{gam2} or \eqref{gam1squ}] with $\mathcal S$ being totally comparable gives
\begin{equation}\label{ngam2}
\left(  G_n   \succsim \Gamma_{2,n}\text{ or }
 \widetilde\rho_{0,n}   \succsim \Gamma_{2,n}\text{ or } 
 \Gamma_{1,n}\widetilde\rho_{1,n}  \succsim \Gamma_{2,n}\right)
\end{equation}
and 
\begin{equation}\label{ngam1}
 \left( G_n \succsim \Gamma_{1,n}^2\text{ or }
  \widetilde\rho_{0,n}\succsim \Gamma_{1,n}^2\text{ or }
  \widetilde\rho_{1,n}\succsim \Gamma_{1,n}\right).
\end{equation}
(Above, we used the fact that the last limit in \eqref{gam1squ} is equivalent to $\lim_{n\to\infty} \widetilde\rho_{1,n}\Gamma_{1,n}/\Gamma_{1,n}^2=0$. Hence, its negation, under the comparability assumption (S4),  is $\widetilde\rho_{1,n}\Gamma_{1,n}\succsim \Gamma_{1,n}^2$ which is equivalent to the last relation in \eqref{ngam1}.)
Applying \eqref{ngam1}, particularly the first two cases in there, along with \eqref{Gnrate} and \eqref{rhrate}, one has
\begin{equation*}
\Gamma_{1,n}=\mathcal O(\lambda_{\varphi(n)+1}^{(\beta-\alpha_*)/2})\text{ as }n\to\infty.
\end{equation*}
Combining \eqref{ngam2} with  \eqref{Gnrate} and \eqref{rhrate}, one has
\begin{equation*}
\Gamma_{2,n}=\mathcal O(\lambda_{\varphi(n)+1}^{\beta-\alpha_*})\text{ as }n\to\infty.
\end{equation*}
Using the last two estimates together with 
\begin{equation}\label{vvv}
     \|v_n-v\|_{D(A^{1+\beta})}=\mathcal O(\Gamma_{1,n}) \text{ and }
     \|v_n-v-\Gamma_{1,n}w_1\|_{D(A^{1+\beta})}     =\mathcal O(\Gamma_{2,n}),
\end{equation}
we obtain \eqref{rate3} and \eqref{rate4}.

(ii) The negation of \eqref{Gtr1} gives $G_n\succsim \Gamma_{1,n}$ or $\widetilde\rho_{0,n}\succsim \Gamma_{1,n}$. In either case, using the first estimate in \eqref{vvv} together with \eqref{Gnrate} and \eqref{rhrate}, we obtain \eqref{rate1}.
\end{proof}

\begin{obs} If $f\in P_NH$ for some $N\ge 1$, then $Q_{\varphi(n)}f=0$ for all $n\ge N$, and we disregard $Q_{\varphi(n)}f$ and $G_n$ completely in Theorem \ref{E4}. 
The same treatment can be applied to the pair
 $Q_{\varphi(n)}B(v,v)$ and $\widetilde\rho_{0,n}$ if $B(v,v)\in P_{N_1}H$, as well as to the pair $Q_{\varphi(n)}\Bs(v,w_1)$ and $\widetilde\rho_{1,n}$ if $\Bs(v,w_1)\in P_{N_1}H$.
\end{obs}

\begin{obs}
If $v$ is such that the mapping 
\begin{equation}\label{lin1to1}
w\in D(A)\mapsto Aw+\Bs (v,w)\in H \text{ is one-to-one,}    
\end{equation}
 then $w_1=0$ in \eqref{rel1} and we obtain the degenerate expansion \eqref{vndeg}.
In particular, if $|f|$ is sufficiently small, then $v$ is unique with sufficiently small $|Av|$ and \eqref{lin1to1} is met.
\end{obs}

\section{Time-dependent solutions}\label{timesec}
We study the convergence of the solution $u_n(x,t)$ of the Galerkin NSE to the weak solution of the NSE.
The key idea is to modify the Aubin--Lions lemma so that compactness can be characterized in terms of similar norms involving fractional derivatives in time. With such a norm-characterization, we can build nested spaces and apply the general results in \hilite{Section} \ref{nestedsec}.

\subsection{A compact embedding lemma}
We follow \cite{TemamAMSbook} with some additional details. 
Let $X$ be a complex Hilbert space.
For $f\in L^1(\mathbb R,X)\cap L^2(\mathbb R,X)$, define the Fourier transform $\mathcal F_X[f]:\mathbb R\to X$ by 
\begin{equation}\label{FX}
   \mathcal F_X[f](\tau)= \widehat f(\tau)=\int_{-\infty}^\infty e^{2\pi i \tau t}f(t){\rm d}t\text{ for }\tau\in \mathbb R.
\end{equation}
 Thanks to Parseval's identity, we can extend continuously the above $\mathcal F_X$ to an Hilbert-space isomorphism $\mathcal F_X:L^2(\mathbb R,X)\to L^2(\mathbb R,X)$.

Let $X$ and $Y$ be two complex Hilbert spaces with 
\begin{equation}\label{cemd}
    \text{the continuous embedding }X\subset Y,
\end{equation}
 and, for $\gamma\in[0,1]$, define
\begin{equation}\label{Hgdef}
    \mathcal H^\gamma(\mathbb R;X,Y)=\left\{
    f\in L^2(\mathbb R;X): |\tau|^\gamma  \mathcal F_Y[f](\tau) \in L^2(\mathbb R;Y)\right\}.
\end{equation}
Thanks to \eqref{cemd}, we have
 $\mathcal F_Y[f](\tau)=\mathcal F_X[f](\tau)$ in \eqref{Hgdef}.
One can characterize the last condition in \eqref{Hgdef} by the existence of the  fractional derivative in time $D_{t,Y}^\gamma f\in L^2(\mathbb R;Y)$, where  $D_{t,Y}^\gamma f$ is 
 the inverse Fourier transform of $(2\pi i \tau)^\gamma \mathcal F_Y [f](\tau)$ in $L^2(\mathbb R;Y)$.
Recall that $\mathcal H^\gamma(\mathbb R;X,Y)$ is a set of equivalence classes of functions with respect to the relation ``equal almost everywhere (a.e.)".

For $f,g\in \mathcal H^\gamma(\mathbb R;X,Y)$, define the inner product
\begin{equation}\label{Hginn}
    \langle f,g\rangle_{\mathcal H^\gamma(\mathbb R;X,Y)} =     \langle  f, g\rangle_{L^2(\mathbb R;X)}
    +\int_{-\infty}^\infty |\tau|^{2\gamma} \langle  \widehat  f(\tau), \widehat g(\tau)\rangle_Y {\rm d}\tau,
\end{equation}
where $\widehat f(\tau)=\mathcal F_X[f](\tau)=\mathcal F_Y[f](\tau)$ and $\widehat g(\tau)=\mathcal F_X[g](\tau)=\mathcal F_Y[g](\tau)$.
The induced norm is 
\begin{equation}\label{Hgnorm}
    \|f\|_{\mathcal H^\gamma(\mathbb R;X,Y)} =   \left( \|f\|_{L^2(\mathbb R;X)}^2+\left \| |\tau|^{\gamma}  \widehat  f(\tau)\right\|_{L^2(\mathbb R;Y)}^2\right)^{1/2}.
\end{equation}
Then $\mathcal H^\gamma(\mathbb R;X,Y)$ with \eqref{Hginn} is a complex Hilbert space.

Let $T>0$. For a function $f:[0,T]\to X$, let 
\begin{equation}\label{zerox}
    \text{$\widetilde f:\mathbb R\to X$ be the extension of $f$ to be $0$ outside of the interval $[0,T]$.}
\end{equation}
For $\gamma\in[0,1]$, define
\begin{equation}\label{xH}
    \mathcal H^\gamma(0,T;X,Y)=\{f:[0,T]\to X\text{ such that }
    \widetilde f\in  \mathcal H^\gamma(\mathbb R;X,Y)\},
\end{equation}
and, for $f,g\in \mathcal H^\gamma(0,T;X,Y)$,
\begin{equation}\label{exnorm}
 \langle f,g\rangle_{\mathcal H^\gamma(0,T;X,Y)}=\langle \widetilde f,\widetilde g\rangle_{\mathcal H^\gamma(\mathbb R;X,Y)},\quad 
    \|f\|_{\mathcal H^\gamma(0,T;X,Y)}=\|\widetilde f\|_{\mathcal H^\gamma(\mathbb R;X,Y)}.
\end{equation}
Then $\mathcal H^\gamma(0,T;X,Y)$ in \eqref{xH} with \eqref{exnorm} is a complex Hilbert space.

By \eqref{cemd}, there is $C_{X,Y}>0$ such that
\begin{equation}\label{cxy}
    \|u\|_Y\le C_{X,Y}\|u\|_X\text{ for all }u\in X.
\end{equation}
This clearly implies 
\begin{equation}\label{LLrel}
\|f\|_{L^2(\mathbb R;Y)}\le C_{X,Y}\|f\|_{L^2(\mathbb R;X)} \ \forall f\in L^2(\mathbb R;X), 
\end{equation}
and, by combining with \eqref{Hgnorm},
\begin{equation}\label{LHrel}
\|f\|_{L^2(\mathbb R;X)}\le\|f\|_{\mathcal H^\gamma(\mathbb R;X,Y)}
\text{ and } 
\|f\|_{L^2(\mathbb R;Y)}\le  C_{X,Y}\|f\|_{\mathcal H^\gamma(\mathbb R;X,Y)}
\ \forall f\in \mathcal H^\gamma(\mathbb R;X,Y).
\end{equation}

For $\gamma,\beta\in[0,1]$ with $\gamma>\beta$, we have
\begin{align*}
   &\left\| |\tau|^\beta  \widehat  f(\tau)\right\|_{L^2(\mathbb R;Y)}^2
   =  \int_{-\infty}^\infty |\tau|^{2\beta}\|\widehat  f(\tau)\|_{Y}^2 {\rm d}t \le    \int_{-\infty}^\infty (1+|\tau|^{2\gamma} ) \|\widehat  f(\tau)\|_{Y}^2 {\rm d}t \\
   &=   \left\| \widehat  f(\tau)\right\|_{L^2(\mathbb R;Y)}^2 +  \left \| |\tau|^{\gamma}  \widehat  f(\tau)\right\|_{L^2(\mathbb R;Y)}^2
   \le C_{X,Y}^2 \|   f(\tau)\|_{L^2(\mathbb R;X)}^2 +  \left \| |\tau|^{\gamma}  \widehat  f(\tau)\right\|_{L^2(\mathbb R;Y)}^2.
\end{align*}
Hence,
\begin{equation*}
      \|f\|_{L^2(\mathbb R;X)}^2+\left \| |\tau|^\beta  \widehat  f(\tau)\right\|_{L^2(\mathbb R;Y)}^2
    \le (C_{X,Y}^2+1)\left(\|   f(\tau)\|_{L^2(\mathbb R;X)}^2 +  \left \| |\tau|^{\gamma}  \widehat  f(\tau)\right\|_{L^2(\mathbb R;Y)}^2\right).
\end{equation*}
Taking the square root of both sides gives
\begin{equation}\label{lower}
\|f\|_{\mathcal H^\beta(\mathbb R;X,Y)}\le \sqrt{C_{X,Y}^2+1}\|f\|_{\mathcal H^\gamma(\mathbb R;X,Y)} \ \forall f\in \mathcal H^\gamma(\mathbb R;X,Y).
\end{equation}

Let $f$ and $\widetilde f$ be as in \eqref{zerox}.
By applying \eqref{LLrel}, \eqref{LHrel} and \eqref{lower} to $f:=\widetilde f$, we obtain similar inequalities 
\begin{equation}\label{LLrel2}
\|f\|_{L^2(0,T;Y)}\le C_{X,Y} \|f\|_{L^2(0,T;X)}\ \forall f\in L^2(0,T;X), 
\end{equation}
\begin{align}\label{LHrel2}
\|f\|_{L^2(0,T;X)}&\le\|f\|_{\mathcal H^\gamma(0,T;X,Y)}, \ 
\|f\|_{L^2(0,T;Y)}\le C_{X,Y}\|f\|_{\mathcal H^\gamma(0,T;X,Y)},\\
\label{lower2}
\|f\|_{\mathcal H^\beta(0,T;X,Y)}&\le \sqrt{C_{X,Y}^2+1}\|f\|_{\mathcal H^\gamma(0,T;X,Y)}
\ \forall f\in \mathcal H^\gamma(\mathbb R;X,Y).
\end{align}

\medskip
We deal with real Hilbert spaces next.
We recall that the complexification of a real  inner product space $(X,\langle \cdot,\cdot\rangle_X)$ is the space $X_{\mathbb C}=\{u+iv:u,v\in X\}$, where $i=\sqrt{-1}$, with the inner product
\begin{equation*}
    \langle u_1+iv_1,u_2+iv_2\rangle_{X_{\mathbb C}}=\langle u_1,u_2\rangle_X+\langle v_1,v_2\rangle_X+i(\langle v_1,u_2\rangle_X-\langle u_1,v_2\rangle_X),
\end{equation*}
for any  $u_1,u_2,v_1,v_2\in X$. The conjugacy in $X_{\mathbb C}$ is defined by 
\begin{equation*}
    \overline{u+iv}=u-iv\text{ for all } u,v\in X.
\end{equation*}
One can verify that 
\begin{equation}\label{Xconj}
    \langle \bar u,\bar v\rangle_{X_{\mathbb C}} =\overline{ \langle u,v\rangle_{X_{\mathbb C}} } \text{ for all } u,v\in X_{\mathbb C}.
\end{equation}
For a function $f\in L^2(\mathbb R,X)$, we have $f\in L^2(\mathbb R,X_{\mathbb C})$ and its Fourier transform $\widehat f(\tau)$ in $X_{\mathbb C}$ satisfies 
\begin{equation}\label{Fconj}
    \widehat f(-\tau) = \overline{\widehat f(\tau)} \text{ a.e. in }\tau.
\end{equation}
The property \eqref{Fconj} can be proved first by using formula \eqref{FX} for all $\tau\in\mathbb R$ and $f\in L^1(\mathbb R,X)\cap L^2(\mathbb R,X)$. It then is extended to $f\in L^2(\mathbb R,X)$ for almost all $\tau\in\mathbb R$ by the use of the continuous extension of $\mathcal F_{X_\mathbb C}$, and the fact that the convergence of a sequence in $L^2(\mathbb R,X_\mathbb C)$ implies the point-wise convergence a.e. in $\mathbb R$.

Now, consider two real Hilbert spaces $X$ and $Y$ satisfying \eqref{cemd} and \eqref{cxy}.
Let $X_{\mathbb C}$ and $Y_{\mathbb C}$ be the complexifications of $X$ and $Y$, respectively. Then $X_{\mathbb C}$ and $Y_{\mathbb C}$ are complex Hilbert spaces with 
\begin{equation}\label{xyC}
        \|u\|_{Y_{\mathbb C}}\le C_{X,Y}\|u\|_{X_{\mathbb C}}\text{ for all }u\in X_{\mathbb C}.
\end{equation}
Thus, $X_{\mathbb C}\subset Y_{\mathbb C}$ is a continuous embedding.

Let $\gamma\in[0,1]$, define
\begin{equation}\label{reHg}
    \mathcal H^\gamma(\mathbb R;X,Y)=\{f:\mathbb R\to X\text{ such that }f\in \mathcal H^\gamma(\mathbb R;X_{\mathbb C},Y_{\mathbb C})\},\ 
\end{equation}
and, for $f,g\in \mathcal H^\gamma(\mathbb R;X_{\mathbb C},Y_{\mathbb C})$,
\begin{equation}\label{reHnorm}
 \langle f,g\rangle_{\mathcal H^\gamma(\mathbb R;X,Y)}=\langle f,g\rangle_{\mathcal H^\gamma(\mathbb R;X_{\mathbb C},Y_{\mathbb C})},\quad 
    \|f\|_{\mathcal H^\gamma(0,T;X,Y)}=\|f\|_{\mathcal H^\gamma(\mathbb R;X_{\mathbb C},Y_{\mathbb C})}.
\end{equation}

We define $\mathcal H^\gamma(0,T;X,Y)$,  $\langle \cdot ,\cdot \rangle_{\mathcal H^\gamma(0,T;X,Y)}$ and $\| \cdot\|_{\mathcal H^\gamma(\mathbb R;X,Y)}$ in the same way as \eqref{zerox}, \eqref{xH}  and \eqref{exnorm}.

Let $f,g\in H^\gamma(\mathbb R;X,Y)$.
Using \eqref{Xconj}  and \eqref{Fconj}, we can prove 
\begin{equation*}
    \int_{-\infty}^\infty |\tau|^{2\gamma}\langle \widehat f(\tau),\widehat g(\tau) \rangle_{Y_{\mathcal C}} {\rm d}\tau \in \mathbb R .
\end{equation*}
Thus, $\langle f,g\rangle_{\mathcal H^\gamma(\mathbb R;X_{\mathbb C},Y_{\mathbb C})} \in\mathbb R$. 
Then we can establish that $H^\gamma(\mathbb R;X,Y)$ in \eqref{reHg} with 
$\langle \cdot ,\cdot \rangle_{\mathcal H^\gamma(\mathbb R;X,Y)}$ in \eqref{reHnorm} is a real Hilbert space. 
As a consequence, we also have $\mathcal H^\gamma(0,T;X,Y)$ is a real Hilbert space.

Thanks to the basic property \eqref{xyC}, all inequalities 
\begin{equation}\label{realtrue}
    \text{\eqref{LLrel}, \eqref{LHrel}, \eqref{lower} and \eqref{LLrel2}, \eqref{LHrel2}, \eqref{lower2} still hold true}
\end{equation} for real Hilbert spaces  $X$ and $Y$.

\medskip
Consider the general case $\mathbb K=\mathbb R$ or $\mathbb C$,  and $X$, $Y$ are two Hilbert spaces over $\mathbb K$ that satisfy \eqref{cemd} and \eqref{cxy}.
Let $I=\mathbb R$ or $[0,T]$.
From \eqref{LLrel}, \eqref{LHrel}, \eqref{LLrel2}, \eqref{LHrel2} and \eqref{realtrue}, we have 
\begin{equation}\label{HXY}
\text{the continuous embeddings } \mathcal H^\gamma(I;X,Y)\subset  L^2(I;X)\subset     L^2(I;Y),
\end{equation}
with the notation $\mathcal H^\gamma([0,T];X,Y)=\mathcal H^\gamma(0,T;X,Y)$.
When $\gamma=0$, the first inclusion in \eqref{HXY} is actually an equality. Indeed, we have
\begin{equation*}
    \|f\|_{\mathcal H^0(I;X,Y)} =   \left( \|f\|_{L^2(I;X)}^2+\left \| f(\tau)\right\|_{L^2(I;Y)}^2\right)^{1/2}.
\end{equation*}
Together with \eqref{cxy} and/or \eqref{xyC}, this implies 
\begin{equation}\label{HzL2} 
\mathcal H^0(I;X,Y)=L^2(I,X)
\text{ with equivalent norms but different inner products.}
\end{equation}
Moreover, \eqref{lower}, \eqref{lower2} and \eqref{realtrue} imply that
 $\mathcal H^\gamma(\mathbb R;X,Y)$ is continuously embedded into $\mathcal H^\beta(\mathbb R;X,Y)$. 
However, we will identify similar spaces that enjoy a compact embedding.

\begin{lem}\label{compact}
Let $\mathbb K$ be either $\mathbb R$  or $\mathbb C$. 
Let $X_0,X,X_1$ be three Hilbert spaces over $\mathbb K$  such that $X_0\Subset X\subset X_1$ which means that the first embedding is compact and the second one is continuous.
    For any numbers $T>0$ and $\gamma,\beta\in[0,1]$ with $\gamma>\beta$, one has the compact embedding
    \begin{equation}\label{HHcom}
        \mathcal H^\gamma(0,T;X_0,X_1)\Subset \mathcal H^\beta(0,T;X,X_1).
    \end{equation}
\end{lem}
\begin{proof}
Consider the case $\mathbb K=\mathbb C$ first.  
 We follow the proof of Theorem 2.2 in Section 2, Chapter 3 of \cite{TemamAMSbook}, pages 186--187 and to help the reader, use similar notation.
 Let $(w_m)_{m=1}^\infty$ be a bounded sequence in $\mathcal H^\gamma(0,T;X_0,X_1)$.
 Let $\widetilde w_m$ be the extension of $w_m$ with $0$ value outside of $[0,T]$.
 Then, by definitions \eqref{xH} and \eqref{exnorm}, $(\widetilde w_m)_{m=1}^\infty$ is  a bounded sequence in $\mathcal H^\gamma(\mathbb R;X_0,X_1)$.
 It follows that there are  a subsequence $(u_\mu)_{\mu=1}^\infty$ of $(\widetilde w_m)_{m=1}^\infty$ and  a function  $w\in  \mathcal H^\gamma(0,T;X_0,X_1)$ such that, with  $u=\widetilde w\in\mathcal H^\gamma(\mathbb R;X_0,X_1)$ as in \eqref{zerox}, one has  $u_\mu\to u$ weakly in both $L^2(\mathbb R;X_0)$ and $\mathcal H^\gamma(\mathbb R;X_0,X_1)$.
 
 Set $v_\mu=u_\mu-u$.  It suffices to prove $v_\mu\to 0$ as $\mu\to\infty$ in $\mathcal H^\beta(\mathbb R;X,X_1)$, which is equivalent to 
\begin{equation}\label{v1}
    v_\mu\to 0 \text{ in $L^2(\mathbb R;X)$ strongly as $\mu\to\infty$,} 
\end{equation}
and
\begin{equation}\label{v2}
    |\tau|^\beta\widehat v_\mu(\tau)\to 0 \text{ in $L^2(\mathbb R;X_1)$ strongly as $\mu\to\infty$.} 
\end{equation}

Note that \eqref{v1} is \cite[(2.31) on page 186]{TemamAMSbook} and was already proved on pages 186--187 of \cite{TemamAMSbook}. 
It suffices to prove \eqref{v2} now.
There is $C>0$ such that 
\begin{equation}\label{vmub}
    \|v_\mu\|_{\mathcal H^\gamma(\mathbb R;X_0,X_1)}\le C\text{ for all }\mu.
\end{equation}
Let 
$I_\mu=\int_{-\infty}^\infty |\tau|^{2\beta}\|\widehat v_\mu(\tau)\|_{X_1}^2{\rm d}\tau.$
For any number $M>0$, we rewrite and estimate $I_\mu$ by 
    \begin{align*}
        I_\mu
        &=\int_{|\tau|<M} |\tau|^{2\beta}\|\widehat v_\mu(\tau)\|_{X_1}^2{\rm d}\tau
            +\int_{|\tau|\ge M} \frac{|\tau|^{2\gamma}}{|\tau|^{2(\gamma-\beta)}}\|\widehat v_\mu(\tau)\|_{X_1}^2{\rm d}\tau\\
        &\le M^{2\beta}J_{\mu,M}  +\frac{1}{M^{2(\gamma-\beta)}}\int_{|\tau|\ge M} |\tau|^{2\gamma}\|\widehat v_\mu(\tau)\|_{X_1}^2{\rm d}\tau,
    \end{align*}
    where
    $J_{\mu,M}=\int_{|\tau|<M} \|\widehat v_\mu(\tau)\|_{X_1}^2{\rm d}\tau$.
    Together with \eqref{vmub}, this implies 
    \begin{equation}\label{Imue}
        I_\mu\le M^{2\beta}J_{\mu,M}  +\frac{1}{M^{2(\gamma-\beta)}} \|v_\mu\|_{\mathcal H^\gamma(\mathbb R;X_0,X_1)}^2
        \le M^{2\beta}J_{\mu,M}  +\frac{C^2}{M^{2(\gamma-\beta)}} .
    \end{equation}
It is proved in \cite[page 187]{TemamAMSbook} that $J_{\mu,M}\to 0$ as $\mu\to\infty$ for any $M>0$. 
Estimating the  limit superior of $I_\mu$ as $\mu\to\infty$ by \eqref{Imue},  and then passing $M\to\infty$, we obtain $I_\mu\to 0$ as $\mu\to\infty$, which proves \eqref{v2}.
This completes the proof of \eqref{HHcom} for the case $\mathbb K=\mathbb C$.

\medskip
Now consider the case $\mathbb K=\mathbb R$.
Let $X_{0,\mathbb C}$, $X_{\mathbb C}$, $X_{1,\mathbb C}$ be the complexifications of $X_0$, $X$, $X_1$, respectively.
    These complexified spaces satisfy $X_{0,\mathbb C}\Subset X_{\mathbb C}\subset X_{1,\mathbb C}.$
    Thanks to the first part of this proof when $\mathbb K=\mathbb C$, we have
       \begin{equation}\label{Cim}
        \mathcal H^\gamma(0,T;X_{0,\mathbb C},X_{1,\mathbb C})\Subset \mathcal H^\beta(0,T;X_{\mathbb C},X_{1,\mathbb C}).
    \end{equation}
    Let $(w_m)_{m=1}^\infty$ be a bounded sequence in  $\mathcal H^\gamma(0,T;X_0,X_1)$. 
\hilite{By the compact embedding \eqref{Cim}, there is a subsequence $(w_{m_j})_{j=1}^\infty$ converging in $\mathcal H^\beta(0,T;X_{\mathbb C},X_{1,\mathbb C})$ as $j\to\infty$ to a function $w\in \mathcal H^\beta(0,T;X_{\mathbb C},X_{1,\mathbb C})$.
As a consequence, $w_{m_j}$ converges in $L^2(0,T;X_{\mathbb C})$ to $w$ as $j\to\infty$. 
Then there is  a further subsequence $(w_{m(\mu)})_{\mu=1}^\infty$ of $(w_{m_j})_{j=1}^\infty$ such  that $w_{m(\mu)}(t)$ converges to $w(t)$ a.e. in $t\in [0,T]$ as $\mu\to\infty$. Since $w_{m(\mu)}(t)\in X$ a.e., we have $w(t)\in X$ a.e., which yields  $w\in \mathcal H^\beta(0,T;X,X_1)$. Therefore, $w_{m_j}$ converges to $w$ in  $\mathcal H^\beta(0,T;X,X_1)$ as $j\to\infty$, and we obtain \eqref{HHcom}.}
\end{proof}

Clearly, Theorem 2.2 in \cite[Section 2, Chapter 3]{TemamAMSbook}
corresponds to the case $\beta=0$ in \eqref{HHcom} thanks to \eqref{HzL2}.

\subsection{Application to the NSE} We will use the estimates (3.31) and (3.33) on page 193 of \cite{TemamAMSbook}, and apply the compact embedding in Lemma \ref{compact} to the spaces
$\mathcal H^{\gamma_k}(0,T; D(A^{\alpha_k}),H)$
for strictly deceasing $\gamma_k$ and $\alpha_k$.

Let $T$, $u_0$, $f$, $u_n$ and $v_n$ be as in the beginning of subsection \ref{maintime}.

\begin{thm}\label{tsoln}
    Let $\alpha_*\in(0,1/2)$ and $\gamma_*\in (0,1/4)$.
    Then there exists a subsequence of $v_n$, still denoted by $v_n$, that possesses a unitary or degenerate expansion
        \begin{equation}\label{vtex1}
        v_n\approx u+\sum_{k\in \mathcal N_2} \Gamma_{k,n} w_k\text{ in } \mathcal H^{\gamma_*}(0,T; D(A^{\alpha_*}),H),
    \end{equation}
    where $u$ is a weak solution the NSE \eqref{fNSE}.
   Consequently, one has a unitary or degenerate expansion
    \begin{equation}\label{vtex2}
        v_n\approx u+\sum_{k\in \mathcal N_2} \bar \Gamma_{k,n} \bar w_k\text{ in } L^2(0,T;D(A^{\alpha_*})).
    \end{equation}
\end{thm}
\begin{proof}
(a) 
Let $$1/2=\alpha_{-1}>\alpha_0>\alpha_1>\alpha_2>\ldots>\alpha_* $$ and $$1/4>\gamma_{-1}>\gamma_0>\gamma_1>\gamma_2>\ldots>\gamma_*.$$
For $k\ge -1$, define
\begin{equation*}
    Z_k=\mathcal H^{\gamma_k}(0,T; D(A^{\alpha_k}),H),\quad Z_*=\mathcal H^{\gamma_*}(0,T; D(A^{\alpha_*}),H).
\end{equation*}
Then each $Z_k$ is a normed space over $\mathbb R$.
For $k\ge -1$, we have $D(A^{\alpha_k})\Subset D(A^{\alpha_{k+1}})$, and hence, by the virtue of Lemma \ref{compact}, 
\begin{equation} \label{comnest}
Z_k\Subset Z_{k+1}\subset Z_*.
\end{equation}

According to property (3.33) with the condition on the power $\gamma\in(0,1/4)$ after (3.38) in  \cite[subsection 3.2]{TemamAMSbook}, $u_n$, and hence $v_n$, are  bounded sequences in $\hilite{\mathcal H}^{\gamma_{-1}}(0,T; V,H)=Z_{-1}$. By \eqref{comnest} with $k=-1$, we have the the compact embedding $Z_{-1}\Subset Z_0$, hence there is a subsequence of $v_n$, still denoted by $v_n$, such that $v_n$ converges to $v$ in $Z_0$.
Thanks to \eqref{compact}, we can apply Theorem \ref{maincor} to obtain a subsequence, still denoted by $v_n$, which has a unitary or degenerate expansion
    \begin{equation}\label{vtex3}
        v_n\approx v+\sum_{k\in \mathcal N_2}^\infty \Gamma_{\hilite{k,n}} w_k\text{ in } Z_*.
    \end{equation}
Observe that $v_n\to u$ in $L^2(0,T;H)$ and,  from \eqref{vtex3},  $v_n\to v$ in $Z_*$ which is continuously embedded in $L^2(0,T;H)$, see \eqref{LHrel2} and \eqref{realtrue}. Thus, we have $u=v$ and obtain \eqref{vtex1} from \eqref{vtex3}.

(b) Let $Z=L^2(0,T;D(A^{\alpha_*}))$. 
Applying the first inequality in \eqref{LHrel2}, see \eqref{realtrue}, to $X=D(A^{\alpha_*})$ and $Y=H$, we have
\begin{equation}\label{LHZ}
    \|w\|_Z\le \|w\|_{Z_*} \text{ for all }w\in Z_*.
\end{equation}
We normalize the asymptotic expansion \eqref{vtex3} in $Z$ as follows. Suppose, corresponding to \eqref{vtex3}, we have
\begin{equation*}
    v_n=u+\sum_{j=1}^{k-1} \Gamma_{j,n}w_j+w_{k,n}.
\end{equation*}
Let $ \bar \Gamma_{n,k}=\|w_k\|_{Z}\Gamma_{n,k}$,  $\bar w_k=w_k  / \|w_k\|_{Z}$, and $\bar w_{k,n}=w_{k,n}  / \|w_k\|_{Z}$
whenever $w_k\ne 0$. 
Then we obtain from \eqref{vtex3} and \eqref{LHZ} the  unitary or degenerate expansion
\begin{equation}\label{vtex4}
        v_n\approx u+\sum_{k\in \mathcal N_2} \bar \Gamma_{k,n} \bar w_k\text{ in } Z,
    \end{equation}
which is  \eqref{vtex2}.
\end{proof}

\begin{proof}[Proof of Theorem \ref{tsoln0}]
    In the proof part (b) of Theorem \ref{tsoln}, we instead take $Z=L^2(0,T;H)$. Then we still have inequality \eqref{LHZ} and obtain the unitary or degenerate expansion \eqref{vtex4} which proves \eqref{vtex0}.
\end{proof}

\begin{obs}
To find more properties of the asymptotic expansion \eqref{vtex1}, better estimates for ${\rm d}u_m/{\rm d}t$ and (probably) more regularity for $f$ are  needed. We will deal with these issues in another work. 
\end{obs}

\appendix

\section{Proofs of some basic inequalities}\label{ProofAB}

Below, the Lebesgue and Sobolev norms are meant to be on $\Omega$.

\begin{proof}[Proof of the second inequality in \eqref{AHm}]
   We have from \eqref{AP} and  \eqref{Preg} that
\begin{equation}\label{Ak1}
\|Au\|_{H^k(\Omega)^d}=\|\mathcal P \Delta u\|_{H^k(\Omega)^d}\le C_k\|\Delta u\|_{H^k(\Omega)^d}\le C_k\|u\|_{H^{k+2}(\Omega)^d}.
\end{equation}
We claim that  
\begin{equation}\label{Ak2}
\|A^m u\|_{H^k(\Omega)^d}\le C_{m,k}\|u\|_{H^{2m+k}(\Omega)^d} \text{ for all } m,k\in\mathbb N\cup\{0\}.
\end{equation}
We prove \eqref{Ak2} by induction in $m$.
The inequality \eqref{Ak2} holds for $m=0$ trivially, and for $m=1$ thanks to \eqref{Ak1}.
Let $m\ge 1$, suppose 
\begin{equation}\label{Ak5}
    \|A^m u\|_{H^k}\le C_{m,k}\|u\|_{H^{2m+k}} \text{ for all } k\in\mathbb N\cup\{0\}.
\end{equation}
Then we have $\|A^{m+1} u\|_{H^k}=\|A(A^{m} u)\|_{H^k}$, and after applying \eqref{Ak1} and then \eqref{Ak5},
\begin{equation*}
\|A^{m+1} u\|_{H^k}\le C_{1,k}\|A^m u\|_{H^{k+2}}\le C_{1,k}C_{m,k+2}\|u\|_{H^{2m+k+2}}=C_{m+1,k}\|u\|_{H^{2(m+1)+k}} 
\end{equation*}
with $C_{m+1,k}=C_{1,k}C_{m,k+2}$.
Thus, by the Induction Principle, \eqref{Ak2} is true.

With $k=0$ in \eqref{Ak2} we obtain the second inequality in \eqref{AHm} for even numbers $m\ge 0$.
Consider  $m=2j+1$ with an integer $j\ge0$. We have from \eqref{Vnorm}, \eqref{VH1} and \eqref{Ak2} that
\begin{equation*}
|A^{m/2} u|=|A^{j+1/2} u|=|A^{1/2}(A^{j} u)|\le \|A^j u\|_{H^1}\le C\|u\|_{H^{2j+1}}=C\|u\|_{H^{m}}. 
\end{equation*}
Thus the second inequality in \eqref{AHm} is true for all integers $m\ge 0$.
\end{proof}

\begin{proof}[Proof of inequality \eqref{AjB}]
The symbol $C$ below denotes a generic positive constant with varying values.
Consider  $u,v\in P_n H$ first. 
By \eqref{BP}, \eqref{AHm} and \eqref{Preg}, we have
\begin{equation}\label{ABpre}
    |A^{m/2} B(u,v)|=|A^{m/2} (\mathcal P(u\cdot\nabla v))|\le  C \|\mathcal P(u\cdot\nabla v)\|_{H^m}\le C \|u\cdot \nabla v\|_{H^m}.
\end{equation}
We estimate the last norm of the product next. 
For $k\in \mathbb N\cup\{0\}$, we denote 
\begin{equation*}
    |D^k u(\bx)|=\sum_{|\omega|=k} |D^\omega u(\bx)|, \text{ where $\omega$ denotes a multi-index,}
\end{equation*}
and $|D^\omega u(\bx)|$ is the Euclidean norm in $\mathbb R^d$.
By the product rule, we have
\begin{equation}\label{Dkuu}
    |D^k (u\cdot\nabla \hilite{v})(\bx)|\le C\sum_{j=0}^k |D^j u(\bx)|\cdot |D^{k-j+1}v(\bx)|.
\end{equation}
If $k\le m-1$, then \eqref{Dkuu} yields
\begin{equation}\label{Dk1}
    |D^k (u\cdot\nabla \hilite{v})(\bx)|\le C\left(\sum_{j=0}^{m-1}|D^{j}u(\bx)|\right)\left(\sum_{\ell=0}^{m}|D^{\ell}v(\bx)|\right)
    \le C\sum_{j,\ell=0}^{m}|D^{j}u(\bx)|\cdot |D^{\ell}v(\bx)|.
\end{equation}
If $k=m$, then by splitting the sum in \eqref{Dkuu} into $j=0$ and $1\le j\le m$, we have
\begin{equation}\label{Dk2}
\begin{aligned}
    |D^m (u\cdot\nabla v)(\bx)|
    &\le C|u|\cdot|D^{m+1}v(\bx)|+C\left(\sum_{j=1}^{m}|D^{j}u(\bx)|\right)\left(\sum_{\ell=1}^{m}|D^{\ell}v(\bx)|\right)\\
        &\le C|u|\cdot|D^{m+1}v(\bx)|+C\sum_{j,\ell=0}^{m}|D^{j}u(\bx)|\cdot |D^{\ell}v(\bx)|.
\end{aligned}
\end{equation}
Combining \eqref{Dk1} and \eqref{Dk2} give
\begin{equation}\label{HuDv}
\|u\cdot\nabla v\|_{H^m}
\le C\left (\left\| \,|u(\bx)|\cdot|D^{m+1}v(\bx)|\,\right\|_{L^2}
+\sum_{j,\ell=0}^{m}\left\|\, |D^{j}u(\bx)|\cdot |D^{\ell}v(\bx)|\,\right\|_{L^2}\right).
\end{equation}
For the first norm on the right-hand side of \eqref{HuDv}, we estimate $|u(\bx)|$ by its $L^\infty$-norm and apply Agmon's inequality.
For the second norm on the right-hand side of \eqref{HuDv},  we apply H\"older's inequality with power $3/2$ for $|D^j u(\bx)|^2$ and power $3$ for $|D^\ell v(\bx)|^2$, then use the interpolation inequality $\|\cdot\|_{L^3}\le \|\cdot\|_{L^2}^{1/2}\|\cdot\|_{L^6}^{1/2}$,  and finally the Gagliardo--Nirenberg--Poincar\'e--Sobolev inequality.
Together with \eqref{ABpre}, they result in
\begin{align*}
    |A^{m/2} B(u,v)|
    &\le  C \|u\|_{H^1}^{1/2}\|u\|_{H^2}^{1/2}\|v\|_{H^{m+1}} +C \sum_{j,\ell=0}^{m}\|\,|D^{j}u|\,\|_{L^3} \|\,|D^{\ell}v|\,\|_{L^6}\\
    &\le  C \|u\|_{H^1}^{1/2}\|u\|_{H^2}^{1/2}\|v\|_{H^{m+1}}+C \sum_{j,\ell=0}^{m}\|\,|D^{j}u|\,\|_{L^2}^{1/2}\|\,|D^{j}u|\,\|_{L^6}^{1/2}\cdot  \|\,|D^{\ell}v|\, \|_{L^6}\\
    &\le  C \|u\|_{H^1}^{1/2}\|u\|_{H^2}^{1/2}\|v\|_{H^{m+1}}+C \|u\|_{H^{m}}^{1/2}\|u\|_{H^{m+1}}^{1/2}\|v\|_{H^{m+1}}.
\end{align*}
Using the fact $m\ge 1$ and estimating the Sobolev norms by the first inequality in \eqref{AHm}, we obtain 
\begin{equation}\label{Bk1}
    |A^{m/2} B(u,v)| 
    \le C \|u\|_{H^{m}}^{1/2}\|u\|_{H^{m+1}}^{1/2}\|v\|_{H^{m+1}}
    \le   C |A^{m/2}u|^{1/2}|A^{(m+1)/2}u|^{1/2}|A^{(m+1)/2}v|.
\end{equation}

For any $u,v\in D(A^{(m+1)/2})$, by applying \eqref{Bk1} to $u:=P_nu$ and $v:=P_n v$ and passing $n\to\infty$, we  obtain \eqref{AjB}.
\end{proof}

\section{Computational example - by Chengzhang Fu}\label{compuapx}

\hilite{The computations below illustrate some of the theoretical findings. They do not exactly follow the theory in the paper as they are done using a pseudo-spectral approximation of the NSE in vorticity form. Specifically, we present several numerical experiments to exemplify two of the cases in Theorem \ref{E3b}. The vorticity formulation reads as}
\begin{align}\label{NSEom}
    \frac {\partial \omega} {\partial t}  -\nu\Delta \omega + \bu \cdot \nabla \omega = g\;,\quad  u =\nabla^{\perp}\psi,\quad -\Delta \psi =\omega
\end{align}
where $g=\nabla \times \bF$.  This means that instead of the restriction of the integer lattice to a disk as for $P_n$, we restrict it to the square $[-n/2,n/2]\times [-n/2,n/2]$.  

We construct a scalar field $\omega(x,y)$ on the periodic domain
$\Omega=[0,2\pi]^2$ as follows.  First, we define the cubic spline
\[
\sigma(s) = 1 - 3s^2 + 2s^3, \qquad s \in [0,1].
\]
This is combined with the even reflection map
\[
\phi(z)=
\begin{cases}
\dfrac{z}{\pi}, & 0 \le z \le \pi, \\[6pt]
\dfrac{2\pi - z}{\pi}, & \pi < z \le 2\pi 
\end{cases}
\]
to define a scalar field $\omega(x,y)$ 
\[
\omega(x,y)
=
\sigma(\phi(x))\, \sigma(\phi(y))
\left(1 + \tfrac{1}{4}\cos(8x)\cos(6y)\right).
\]
This serves will serve as the limit of our sequence.  It is balanced by the force 
\begin{equation} \label{rel:ug}
g(x,y)
=
- \nu \Delta \omega(x,y)
+ \bu(x,y) \cdot \nabla \omega(x,y).
\end{equation}

Since the fast Fourier transform works best with powers of small primes, we take $n=2^p3^q$.
We start from the zero initial condition for $n_1=32$, evolve \eqref{NSEom} until a steady state is (nearly) reached, increase the resolution to the  $n_2=36$ and resume the evolution until a steady state is reached again.  This is repeated through resolution $n_{44}=4096$.
The corresponding physical-space visualizations of $\omega$ and $g$ are shown
in Figure~\ref{Fig:phys}.  The time stepping is done by a 3rd order Adams-Bashforth method with $\Delta t = 10^{-3}$ and the viscosity is taken to be $\nu=0.01$.

\begin{figure}[H]
\centerline{
\includegraphics[width=8.5cm, height=6.5cm]{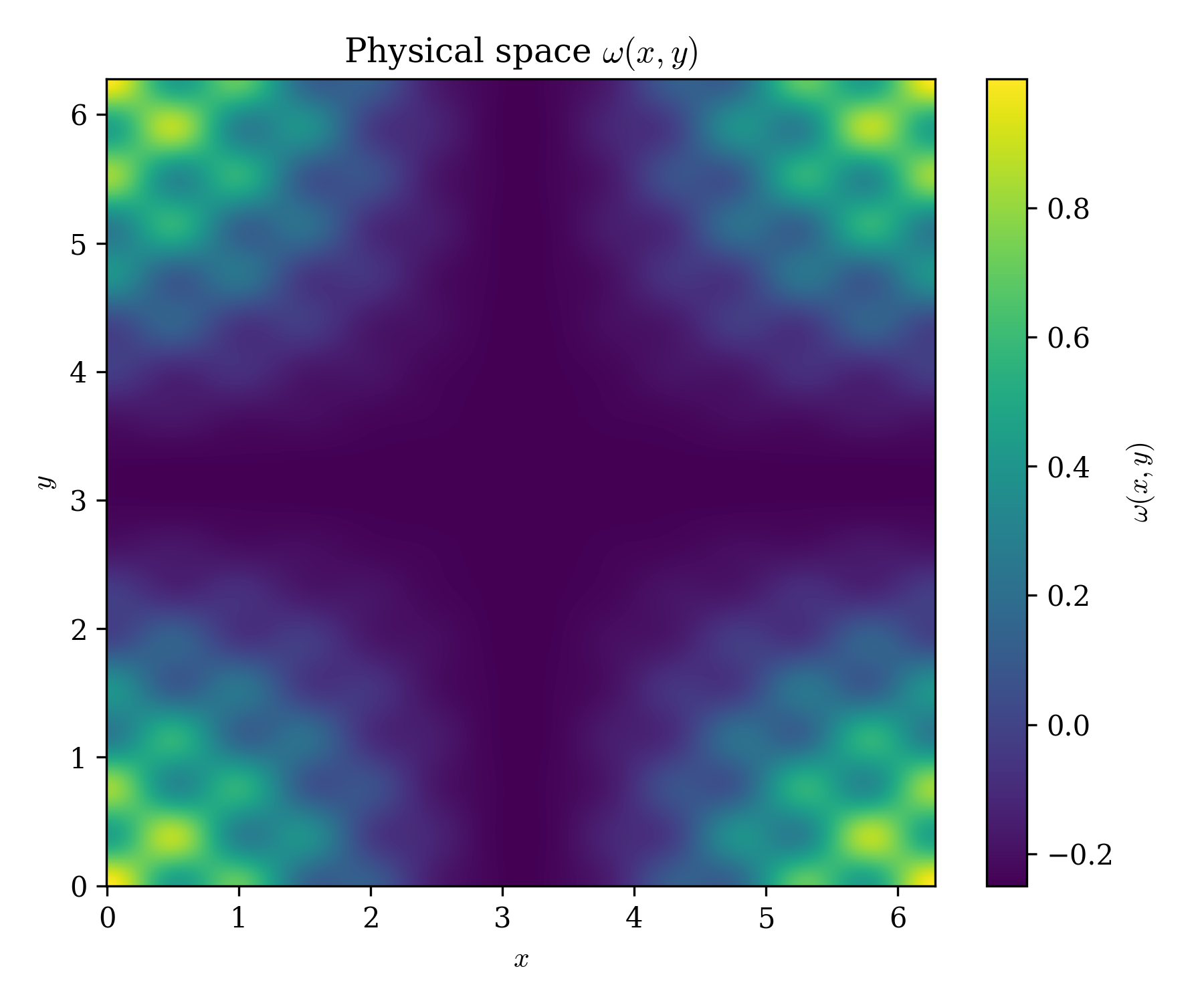}\
\includegraphics[width=8.5cm, height=6.5cm]{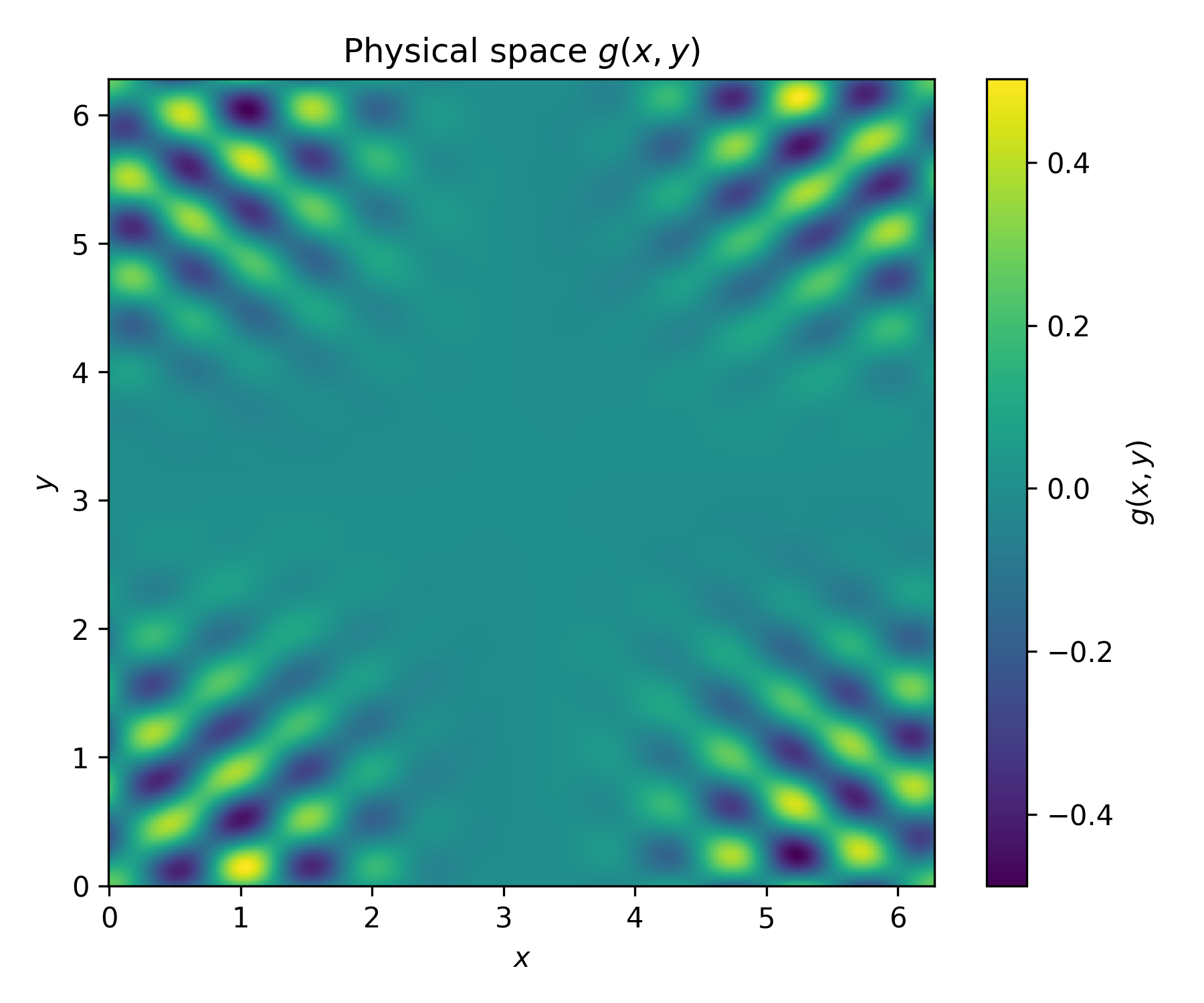}}
\caption{Physical-space plots of the scalar field $\omega$ (left) and the forcing term $g$ (right)}
\label{Fig:phys}
\end{figure}

We define three diagnostic quantities to measure the convergence of the
approximating sequence. Let $\tilde b=\bu(x,y) \cdot \nabla \omega(x,y)$, and denote the corresponding Galerkin (pseudospectral) approximations by
$\omega_n$, $\tilde b_n$, and $g_n$.
Taking $Z_0=D(A^{1/2})$, we recognize the  coefficients in the expansions for the velocity form of the NSE by
\begin{equation*}
\Gamma_{1,n}
= \|\omega_n - \omega\|_{L^2}, \quad
\rho_{1,n}
= \|\tilde b_n - \tilde b\|_{L^2}, \quad
G_{n}
= \|g_n - g\|_{L^2}.
\end{equation*}
The numerical results for $\Gamma_{1,n}$, $\rho_{1,n}$, and $G_{n}$, together
with their corresponding quotients, are shown in 
 Figure~\ref{Fig:conv1}.    With $\Delta t = 10^{-3}$, the Adams-Bashforth method has a
  \hilite{global} truncation error of size $\mathcal{O}(\Delta t^3)\approx 10^{-9}$. This could explain why $\Gamma_{1,n}$ does not get close to machine precision $(10^{-15})$.
  
\begin{figure}[ht]
\centerline{
\includegraphics[width=8.5cm, height=6.5cm]{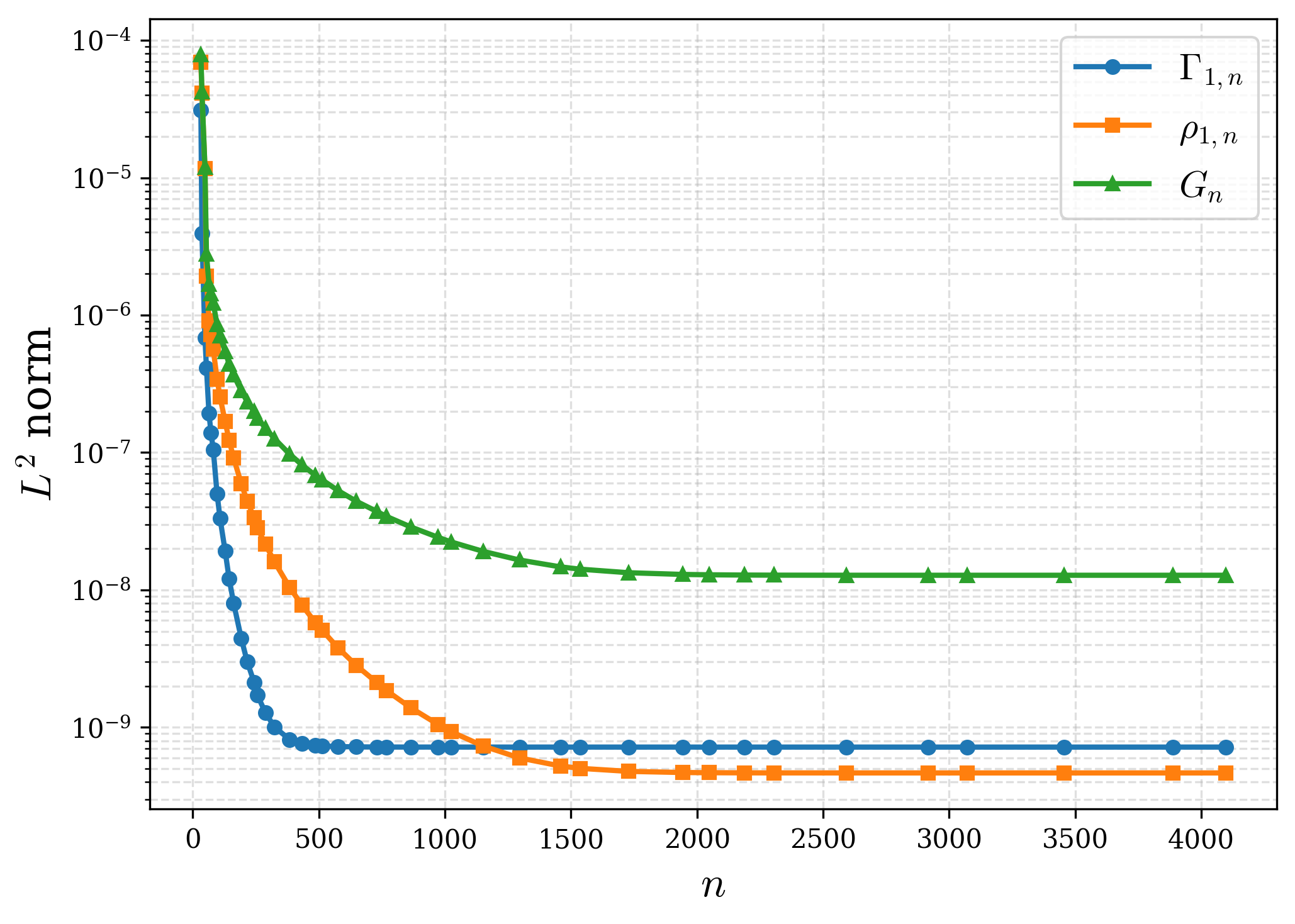}\
\includegraphics[width=8.5cm, height=6.5cm]{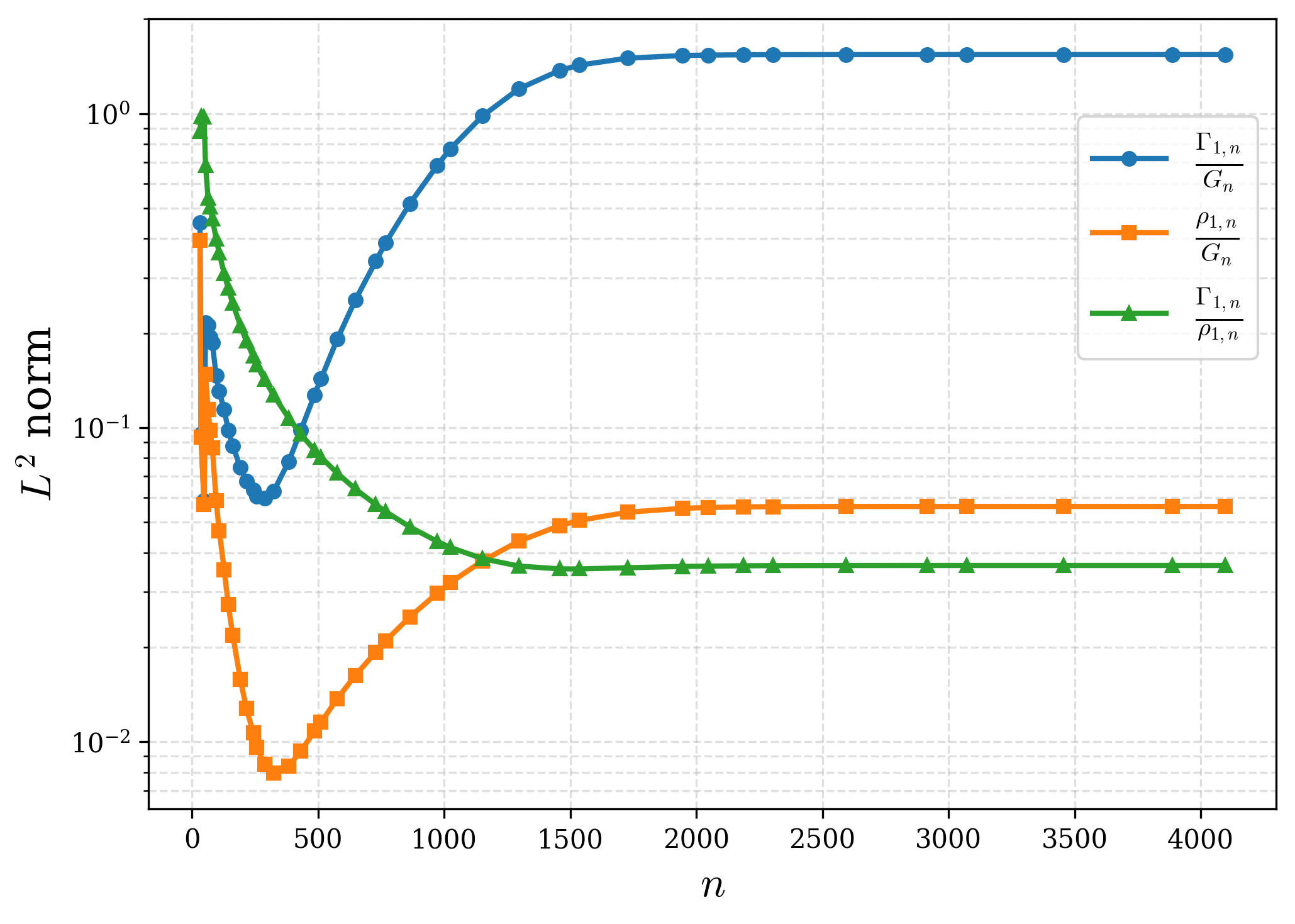}}
\caption{Left: convergence plots for $\Gamma_{1,n}$, $\rho_{1,n}$, and $G_{n}$. Right: the corresponding quotient plots.}
\label{Fig:conv1}
\end{figure}

Since it appears that $\Gamma_{1,n} \sim \rho_{1,n} \sim G_n$ we expect for Case 2 of Theorem \ref{E3b} to hold.  We do not know a sharp estimate for $\alpha_*$ such that $f\in D(A^{\alpha_*})$, but it seems safe to take $\alpha=1/2$. Since $\omega \in C^2$ so $u \in D(A^{3/2}$). 
To demonstrate the convergence rate
\[
\|v_n -v\|_{D(A)}
=
\mathcal{O}\!\left(\lambda_{n+1}^{-\alpha_*}\right).
\]
we plot in Figure \ref{Fig:alphastar_big} (left) the product 
$
\|v_n -v\|_{D(A)}
\cdot
\lambda_{n+1}^{\alpha_*}
$.
We suspect the bend upwards at the higher resolutions is due to the time step being too large.  With $\Delta t = 10^{-3}$, the Adams--Bashforth method has a
  \hilite{global} truncation error of size $\mathcal{O}(\Delta t^3)\approx 10^{-9}$. This could also explain why $\Gamma_{1,n}$ does not get close to machine precision $(10^{-15})$.

\begin{figure}[H]
\centerline{
\includegraphics[width=8.5cm,height=6.5cm]{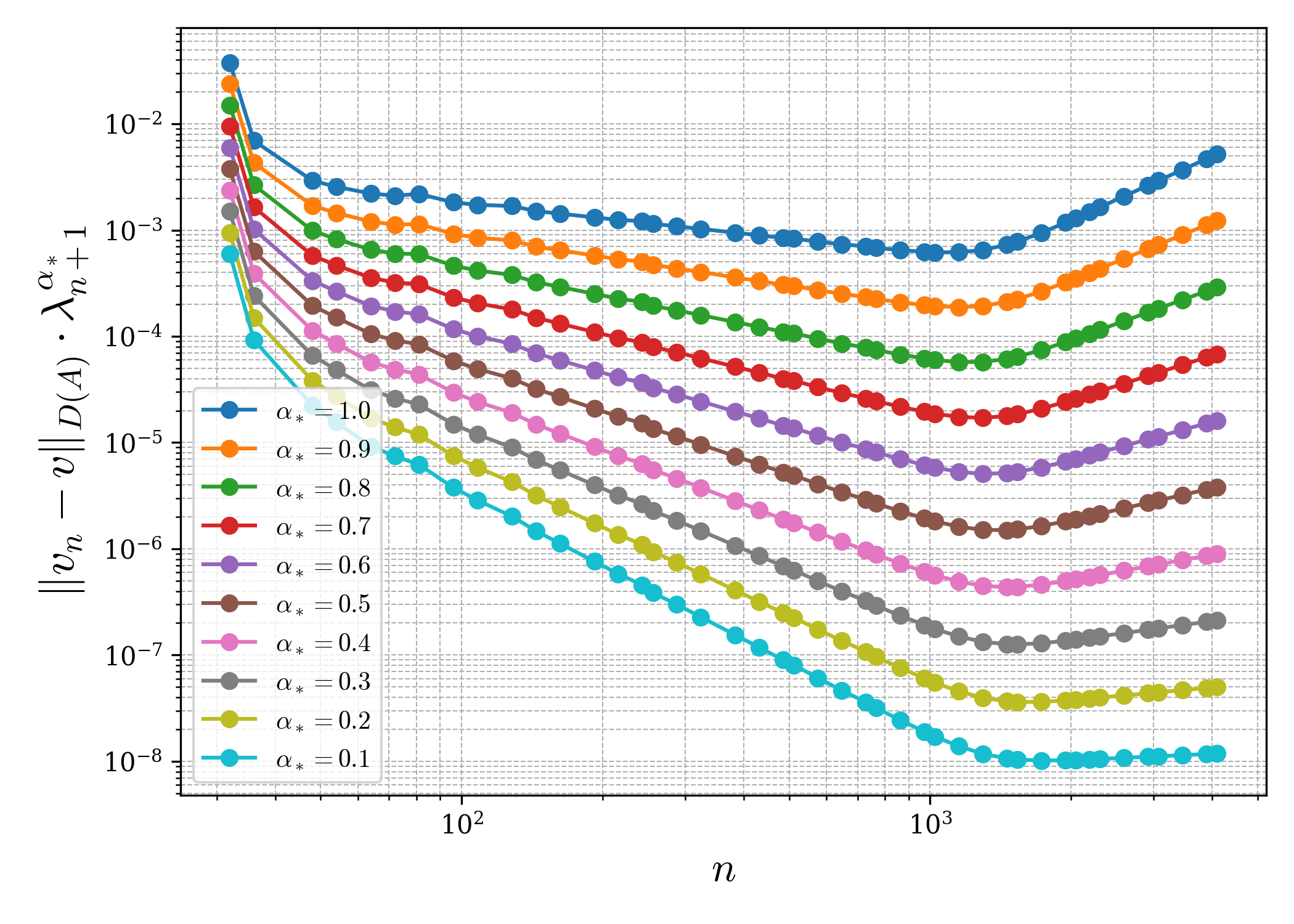}\,
\includegraphics[width=8.5cm,height=6.5cm]{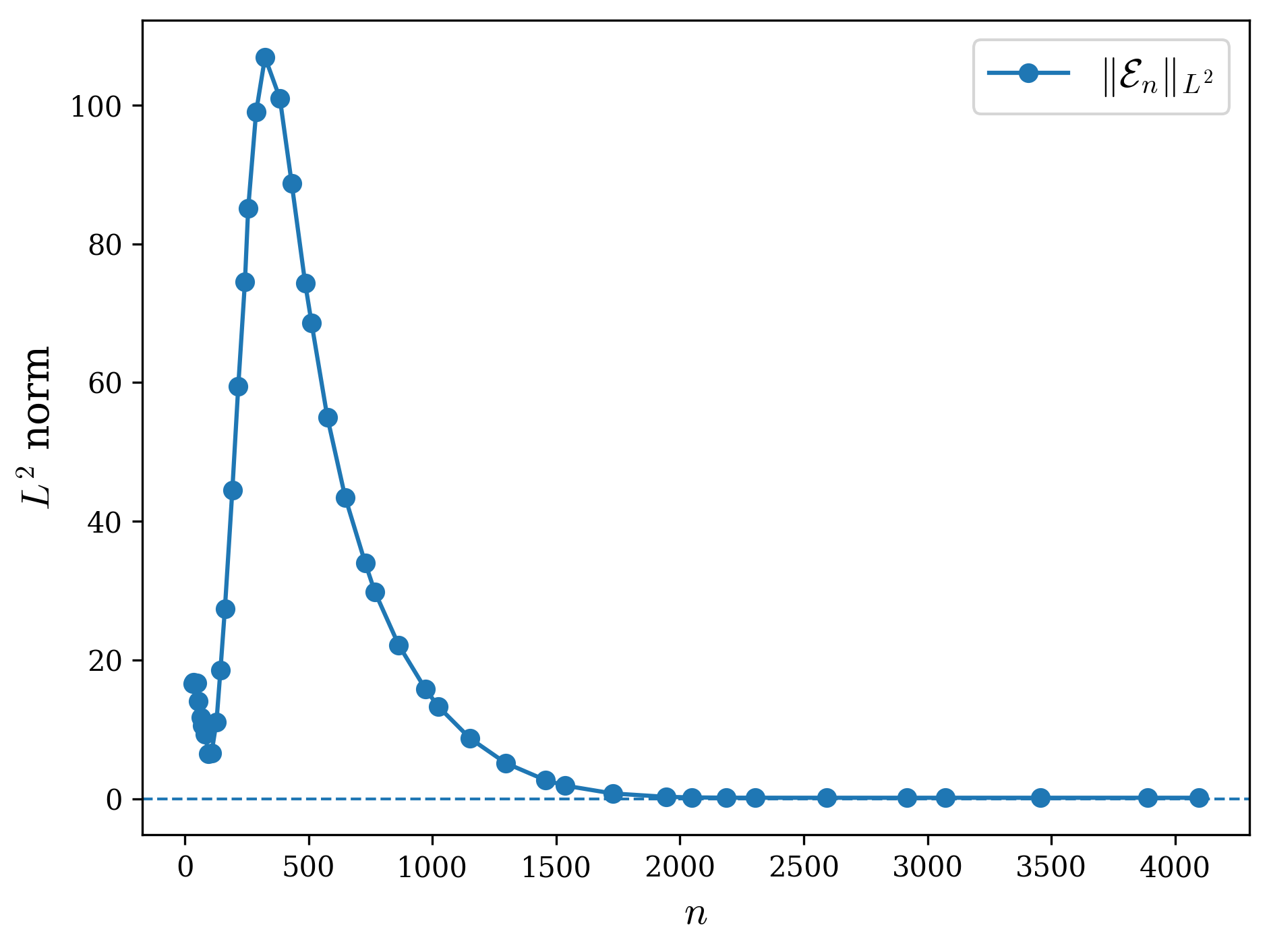}}
\caption{Left: the product 
$
\|v_n -v\|_{D(A)}
\cdot
\lambda_{n+1}^{\alpha_*}
$ for various values of $\alpha_*$.  Right: convergence $\mathcal{E}_n \to 0$, where $\mathcal{E}_n$ is defined in \eqref{EEn}.}
\label{Fig:alphastar_big}
\end{figure}

Consider the normalized vectors
\[
w_n^{(1)} =\frac{v_n - v}{\Gamma_{1,n}}, \qquad b_n^{(1)} =\frac{ b_n - b}{\rho_{1,n}},
\qquad 
f_n^{(1)}=\frac{f_n-f}{G_n} =\frac{Q_n f}{G_n}
\]
In the constructive existence proof for an intrinsic expansion of a convergent sequence in \cite{FHJ} as well as in \cite{HJ1} the first vector can be found directly provided the limit of the normalized vector exists.   In the particular expansions 
$$
v_n \approx v+ \sum \Gamma_{k,n} w_k, \qquad 
b_n \approx b+ \sum \rho_{k,n} b_k, \qquad
f_n \approx f+ \sum f_{k,n} \phi_k
$$
we would have
\[
w_1=\lim_{n \to \infty}w_n^{(1)}, \qquad
b_1=\lim_{n \to \infty}b_n^{(1)},
\qquad 
\phi_1=\lim_{n \to \infty}f_n^{(1)}\;.
\]
Another consequence of $\Gamma_{1,n} \sim \rho_{1,n} \sim G_n$ is that these vectors should satisfy 
$$\nu Aw_1+\mu_0 b_1 = \mu_{00}\phi_1\;,$$
where 
$$\mu_0=\lim_{n\to\infty}\frac{\rho_{1,n}}{\Gamma_{1,n}}\;, \quad 
\mu_{00} =\lim_{n\to\infty}\frac{G_n}{\Gamma_{1,n}}
$$
Note that this is similar to a relation in Case 1 (i) of Theorem \ref{E3b} where $\mu_{00}=0$.
Applying the curl operator, we have for the same coefficients $\mu_0$, $\mu_{00}$
$$\nu A\omega_1 + \mu_0 \tilde b_1 =\mu_{00}\gamma_1$$
where $\omega_1=\nabla \times w_1$, $\tilde b_1=\nabla \times b_1$ and $\gamma_1=\nabla\times \phi_1$.

The convergence 
\begin{equation}\label{EEn}\mathcal{E}_n=\nu A\frac{\omega_n-\omega}{\Gamma_{1,n}} + \frac{\rho_{1,n}}{\Gamma_{1,n}} 
\frac{\tilde b_n -\tilde b}{\rho_{1,n}} - 
\frac{G_{n}}{\Gamma_{1,n}}\frac{Q_ng}{G_n} \to  0  \quad \text{in}\  L^2(\Omega)
\end{equation}
 is demonstrated in Figure \ref{Fig:alphastar_big} (right). 

\bigskip
\noindent\textbf{Methods.} 
No Artificial Intelligence Generated Content (AIGC) tools are used in developing any portion of this paper. 

\medskip
\noindent\textbf{Conflict of interest.}
There are no conflicts of interests.

 \bibliographystyle{abbrv}
 \bibliography{HJ2cite.bib}

\end{document}